\documentclass[12pt]{amsart}
\usepackage{graphicx,amssymb,amsfonts,latexsym,amsmath,amsthm,times}
\usepackage{mathrsfs}
\usepackage{epsfig}
\usepackage{color}
\usepackage[english]{babel}

\numberwithin{equation}{section}
\usepackage{amsmath,amsthm,amssymb,amsfonts}
\usepackage[active]{srcltx}
\usepackage{latexsym}
\usepackage{enumerate}
\usepackage{verbatim}
\usepackage{multicol}
\usepackage{graphicx}
\usepackage{epsfig}
\usepackage{color}
\usepackage{bezier}
\usepackage{curves}
\usepackage[retainorgcmds]{IEEEtrantools}
\newtheorem{theorem}{Theorem}[section]
\newtheorem{lemma}[theorem]{Lemma}
\newtheorem{corollary}[theorem]{Corollary}
\newtheorem{proposition}[theorem]{Proposition}

\newtheorem{conjecture}[theorem]{Conjecture}
\newtheorem{example}[theorem]{Example}

\newtheorem{claim}{\noindent {\bf Claim}}
\newtheorem{question}{\noindent {\bf Question}}

\newtheorem{fact}{\noindent{\bf Fact}}

\newcommand{\NN}{{\mathbb N}}
\newcommand{\ZZ}{{\mathbb Z}}

\def\endproof{\hfill {\kern 6pt\penalty 500
\raise -0pt\hbox{\vrule \vbox to5pt {\hrule width 5pt
\vfill\hrule}\vrule}}}

\title[Classes of ordered structures]{Profile and hereditary classes of ordered relational structures}
\author[D.Oudrar] {Djamila Oudrar}
\thanks{*The author was supported by CMEP-Tassili grant}
\address{Faculty of Mathematics, USTHB, Algiers, Algeria}
\email {dabchiche@usthb.dz}
\author [M.Pouzet]{Maurice Pouzet}
\address{ICJ, Math\'ematiques, Universit\'e
Claude-Bernard Lyon1, 43 Bd. 11 Novembre 1918
F$69622$ Villeurbanne  cedex, France and University of Calgary, Department of Mathematics and Statistics, Calgary, Alberta, Canada T2N 1N4} \email{
pouzet@univ-lyon1.fr }

\date{\today}

\begin{document}



%\maketitle

\begin{abstract}
Let $\mathfrak{C}$ be a class of finite combinatorial structures. The \textit{profile} of $\mathfrak{C}$ is the function $\varphi _{\mathfrak{C}}$ which counts, for every integer $n$, the number
$\varphi _{\mathfrak{C}}(n)$ of members of $\mathfrak{C}$ defined on $n$ elements, isomorphic structures been identified. The \textit{generating function of}
$\mathfrak{C}$ is $\mathcal {H}_{\mathfrak{C}}(x):=\sum_{n\geqq 0}\varphi _{\mathfrak{C}}(n)x^{n}$.
Many results about the behavior of the function $\varphi _{\mathfrak{C}}$ have been obtained. Albert and Atkinson have shown that the generating series of several classes of permutations are algebraic.
In this paper, we  show how their results extend to classes of ordered binary relational structures; putting emphasis on the notion of hereditary well quasi order, we discuss some of their questions and answer  one.
 \end{abstract}
\maketitle

\vspace{.08in} \noindent \textbf{Keywords}:  profile, well quasi-ordering, indecomposability, permutations.

\section{Introduction}
The context of this paper is  the enumeration of finite relational structures.
A relational structure $\mathcal{R}$ is \textit{embeddable} in a relational structure $\mathcal{R'}$, in notation $\mathcal{R}\leq \mathcal{R'}$, if $\mathcal{R}$ is isomorphic to an induced substructure of
 $\mathcal{R'}$. The embeddability relation is a quasi order. Several significant properties of relational structures or classes of relational structures can be uniquely expressed in term of this quasi order.
This is typically the case of hereditary classes: a class $\mathfrak{C}$ of structures is \textit{hereditary} if it contains every relational structure which can be embedded in some member of $\mathfrak{C}$.
 Interesting hereditary classes abound. In the late forties, Fra\"{\i}ss\'{e}, following the work of Cantor, Hausdorff and Sierpinski, pointed out the role of the quasi-ordering of embeddability and hereditary
classes in the theory of relations (see his book \cite{fraisse} for an illustration). Recent years have seen a renewed interest for the study of these classes, particularly those made of finite structures.
Many results have been obtained. Some are about
obstructions allowing to define these classes, others on the  behavior of the function $\varphi _{\mathfrak{C}}$, the \textit{profile} of $\mathfrak{C}$ which counts, for every integer $n$, the number
$\varphi _{\mathfrak{C}}(n)$ of
members of $\mathfrak{C}$ defined on $n$ elements, isomorphic structures being identified. General counting results have been obtained, as well
 as precise results,  for graphs, tournaments and  ordered graphs (see the survey \cite{klazar}). Enumeration results on permutations, motivated by the \textit{Stanley-Wilf conjecture},
solved by Marcus and Tard\"os (2004), fall also under this frame, an  important fact  due to Cameron \cite{cameron}.  Indeed, to each permutation  $\sigma$ of $[n]:=\{1, \dots, n\}$ we may associate the relational structure $C_{\sigma}:=([n], \leq, \leq_{\sigma})$, that we call \emph{bichain},  made of two linear orders on $[n]$ ($\leq$ being the natural order on $[n]$
and $\leq_{\sigma}$ the linear order defined by
$i\leq_{\sigma} j$
if and only if  $\sigma(i) \leq \sigma(j)$). As it turns out, the order defined on permutations and the embeddability between bichains coincide (see Subsection \ref {permutation-bichain} for details and examples).

 In this paper,  we  show how some results obtained by Albert and Atkinson \cite{A-A} for classes of permutations extend to classes of ordered binary relational structures. We prove notably Theorem \ref{theo:algebraic}. For this purpose,  we recall in Section \ref{sect:basic} some basic definitions of the theory of relations, we survey in Section \ref{sect:bichains} some results concerning classes of permutations and show how permutations are related to relational structures. Then, we illustrate  the role
of indecomposable structures (see Section \ref{sect:indecomposable}) and of well quasi order (see Section \ref{sect:well-quasi-order})   in enumeration results. Finally, in Section \ref{sect:conjecture}, we present a conjecture and a partial solution, a special case answering a question of Albert and Atkinson \cite{A-A}.

Our results have been presented at  the international conference on Discrete Mathematics and Computer Science (Dimacos'11)  held in Mohammedia, Morocco, May-5-8, 2011, and at the  International Symposium on Operational Research (Isor'11),  held in Algiers, Algeria , May 30-June 2, 2011 \cite{oudrar-pouzet}. We are pleased to thanks the organizers of these meetings for their help.

\section{Basic notions, embeddability, hereditary classes and profile}\label{sect:basic}
Our terminology agree with \cite{fraisse}.
Let $n$ be a positive integer. A $n$-{ary relation} with \emph{domain} $E$ is a subset $\rho$ of the $n$-th power $E^{n}$ of $E$; for $n=1$ and  $n=2$  we use the words  \emph {unary relation } and
\emph {binary relation}, in this later case we rather set $x\rho y$ instead of $(x,y)\in \rho$. A \textit{relational structure} with \emph{domain} $E$ is a pair $\mathcal{R}:= (E,(\rho_i)_{i\in I})$ made of a set $E$
and a family $(\rho_i)_{i\in I}$ of $n_i$-ary relations
$\rho_i$ on $E$, each $\rho_i$ being a subset of $E^{n_i}$. The family $\mu:=(n_i)_{i\in I}$ is the \emph{signature} of $\mathcal{R}$. We will denote by $V(\mathcal R)$ the domain of $\mathcal R$.  We denote by $\overline\Omega_\mu$ the class of these structures and by $\Omega_\mu$ the subclass of  the finite ones. A relational structure $\mathcal{R}$ is \emph{ordered} if it can be expressed as $\mathcal{R}:=(E,\leq,(\rho_j)_{j\in J})$ where $"\leq"$ is a linear order on $E$  and the $\rho_j$'s are $n_j$-ary relations;
the (truncated) signature in this case is $\mu=(n_j)_{j\in J}$.  A relational structure $\mathcal{R}:=(E,(\rho_i)_{i\in I})$ is a  \emph{binary relational structure}, \emph{binary structure} for short, if each $\rho_i$ is a binary relation; the class of those finite
binary structures will be denoted by $\Omega_I$ instead of $\Omega_\mu.$
Basic examples of ordered binary structures are chains $(J=\emptyset)$,
 bichains ($J=\{1\}$ and $\rho_1$ is a linear order) and multichains ($\rho_j$ is a linear order for all $j\in J$). We denote by $\Theta_d$ the collection of finite ordered binary structures made of a linear order and $d$ binary relations. Let $\mathcal{R}:= (E,(\rho_i)_{i\in I})$ be a relational structure;  the
\emph{substructure induced by  $\mathcal{R}$ on a subset $A$ of $E$},
simply called the \emph{restriction of $\mathcal{R}$ to $A$}, is the relational structure
$R\restriction_A :=(A,({\rho_i}\restriction_A)_{i\in I}),$ where ${\rho_i}\restriction_A:={\rho_i}\cap A^{n_i}.$
Let $\mathcal{R}:= (E,(\rho_i)_{i\in I})$ and $\mathcal{R'}:= (E',(\rho'_i)_{i\in I})$ be two relational structures of the same signature $\mu:=(n_i)_{i\in I}.$ A map $\emph f:E\rightarrow E'$ is an
 \emph{isomorphism from $\mathcal{R}$ onto $\mathcal{R'}$} if $f$ is bijective and $(x_1,\ldots,x_{n_i})\in \rho_i$  if and only if  $(f(x_1),\ldots,f(x_{n_i}))\in \rho'_i$ for every
$(x_1,\ldots,x_{n_i})\in E^{n_i}$, $i\in I.$
The relational structure $\mathcal{R}$ is \emph{isomorphic} to $\mathcal{R'}$ if there is some isomorphism from $\mathcal{R}$ onto $\mathcal{R'}$, it is \emph{embeddable} into $\mathcal{R'}$, and we set
$\mathcal{R}\leq \mathcal{R'}$, if $\mathcal{R}$ is isomorphic to some restriction of $\mathcal{R'}.$ The embeddability relation (called ``abritement'' by  \emph{Fra\"{\i}ss\'{e}} in french) is a quasi-order.
A class $\mathfrak{C}$ of relational structures is \emph{hereditary} if $\mathcal{R} \in \mathfrak{C}$ and $\mathcal{S}\leq \mathcal{R}$ imply $\mathcal{S}\in \mathfrak{C};$ relational structures which are not in
$\mathfrak{C}$ are \emph{obstructions} to $\mathfrak{C}$. The \emph{age}  of a relational structure $\mathcal R$ is the class $Age (\mathcal R)$ of finite $\mathcal S$ which are embeddable into $\mathcal R$ (equivalently, this is the set of finite restrictions of $\mathcal R$ augmented of their isomorphic copies). An age is non-empty, hereditary and up-directed (that is for every $\mathcal S, \mathcal S'\in Age(\mathcal R)$ there is some $\mathcal T \in Age( \mathcal R)$ which embeds $\mathcal S$ and $\mathcal S'$). In the terminology of posets, this is an \emph{ideal} of $\Omega_\mu$. If the signature is finite, every ideal of $\Omega_\mu$ is the age of some relational structure (\emph{Fra\"{\i}ss\'{e}} 1954). If $\mathfrak{B}$ is a subset of $\Omega_\mu$ then $Forb(\mathfrak{B})$ denotes the subclass of members of $\Omega_\mu$ which embed no member of
$\mathfrak{B}$. Clearly, $Forb(\mathfrak{B})$ is an hereditary class. Moreover, every hereditary subclass $\mathfrak{C}$ of $\Omega_\mu$ has this form. This fact, due to \emph{Fra\"{\i}ss\'{e}}, is based on the notion
of bound:  a \textit{bound} of an hereditary subclass $\mathfrak{C}$ of $\Omega_\mu$  is every finite $\mathcal{R}$ not in $\mathfrak{C}$ such that every  $\mathcal{R'}$ which
strictly embeds into $\mathcal{R}$ belongs to $\mathfrak{C}$.
  Clearly, every finite obstruction to $\mathfrak{C}$  contains a bound. Hence, if  $\mathfrak{B(C)}$ denotes the collection
 of  bounds of $\mathfrak{C}$ considered up to isomorphism then $\mathfrak{C}=Forb(\mathfrak{B(C)})$.\\  The \textit{profile} of an hereditary class $\mathfrak{C}$ is the function
$\varphi _{\mathfrak{C}}:\mathbb{N\longrightarrow N}$ which counts, for every $n$,
 the number of members of $\mathfrak{C}$
defined on $n$ elements, isomorphic structures been identified. The \textit{generating function for }$\mathfrak{C}$ is $\mathcal {H}_{\mathfrak{C}}(x):=\sum_{n\geqq 0}\varphi _{\mathfrak{C}}(n)x^{n}$.
These two notions are the specialization to hereditary classes of basic notions in enumeration. Many results on the enumeration of classes of permutations are about the enumeration of relational structures.
Indeed, as mentioned in the introduction, permutations can be considered as  special cases of binary structures,  and more specifically ordered binary structures, in fact bichains. We introduce these notions below and point out the  relationship between permutations and
bichains in the next section.\\
%
%
%\noindent A relational structure $\mathcal{R}$ is \emph{ordered} if it can be expressed as $\mathcal{R}:=(E,\leq,(\rho_j)_{j\in J})$ where $"\leq"$ is a linear order on $E$  and the $\rho_j$'s are $n_j$-ary relations;
%the (truncated) signature in this case is $\mu=(n_j)_{j\in J}.$\\
%
%\noindent A relational structure $\mathcal{R}:=(E,(\rho_i)_{i\in I})$ is a  \emph{binary relational structure}, \emph{binary structure} for short, if each $\rho_i$ is a binary relation; the class of those finite
%binary structures will be denoted by $\Omega_I$ instead of $\Omega_\mu.$
%Basic examples of ordered binary structures are chains $(J=\emptyset)$,
% bichains ($J=\{1\}$ and $\rho_1$ is a linear order) and multichains ($\rho_j$ is a linear order for all $j\in J$). We denote by $\Theta_d$ the collection of finite ordered binary structures made of a linear order and $d$ binary relations.

\section{Permutations, bichains and their profile}\label {sect:bichains}
\subsection{Permutations}
Let $n$ be a non negative integer. Let $\mathfrak S_n$ be the set of permutations on $[n]:=\{1,\dots, n\}$ and $\mathfrak S:=\bigcup_{n\in \mathbb  N} \mathfrak S_n.$ An order relation on
$\mathfrak S$ is defined as follows: the permutation $\pi$ of $[n]$ \emph{contains} the permutation $\sigma$ of $[k]$ and we write $\sigma \leq \pi$  if some subsequence of $\pi$
of length $k$ is order isomorphic to $\sigma$. More precisely,
 $\sigma \leq \pi$ if there exist integers $1\leq x_1 <\dots<x_k\leq n$ such that for $1\leq i,j\leq k,$\\
$$\sigma(i)< \sigma(j)~~\text{if and only if}~~ \pi(x_i)< \pi(x_j).$$
For example, $\pi:=391867452$ contains $\sigma:=51342,$ as it can be seen by considering the subsequence $91672$ ($=\pi(2),\pi(3),\pi(5),\pi(6),\pi(9)$).

A subset $\mathfrak C$ of $\mathfrak S$ is \emph{hereditary} if $\sigma <\pi \in\mathfrak C$ implies $\sigma \in \mathfrak C$. Its counting function, that we call the \emph{profile} of $\mathfrak C$, is
$\varphi _{\mathfrak{C}}(n):=|\mathfrak C\cap \mathfrak S_n|$.
How much does $\varphi _{\mathfrak{C}}(n)$ drop from $\varphi _{\mathfrak{S}}(n)=n!$ if $\mathfrak C \neq \mathfrak S?$ The \emph{Stanley-Wilf} conjecture asserted that it drops to exponential growth.
The conjecture was proved in $2004$ by \emph{Marcus and Tard\"os} \cite{Marcus}:

\begin{theorem}
 If $\mathfrak C$ is a proper hereditary set of permutations, then, for some constant $c,~~\varphi_{\mathfrak{C}}(n)< c^n$ for every $n$.
\end{theorem}

\emph{Kaiser and Klazar} \cite{K-K} proved that if $\mathfrak C$ is hereditary, then, either $\varphi_{\mathfrak{C}}$ is bounded by a polynomial and in this case is a polynomial, or is bounded below by an
exponential, in fact the \emph{generalized Fibonacci function} $F_{n,k}$.\\
We recall that the generalized Fibonacci number is given by the recurrence $F_{n,k}=0$ for $n<0$, $F_{0,k}=1$ and $F_{n,k}=F_{n-1,k}+F_{n-2,k}+\ldots+F_{n-k,k}$ for $n>0$. $F_{n,k}$ is the coefficient of $x^n$
in the power series expansion of the expression $\dfrac{1}{1-x-x^2-\ldots-x^k}.$ The Kaiser and Klazar theorem reads as follow:

\begin{theorem}
If $\mathfrak C$ is an hereditary set of permutations, then exactly one of the four cases occurs.
\begin{enumerate}
\item For large $n$, $\varphi_{\mathfrak{C}}(n)$ is eventually constant.
\item There are integers $a_0,\ldots,a_k,$ $k\geq 1$ and $a_k>0$, such that $\varphi_{\mathfrak{C}}(n)=a_0 \dbinom{n}{0}+\ldots+a_k \dbinom{n}{k}$ for large $n$. Moreover, $\varphi_{\mathfrak{C}}(n)\geq n$ for
 every $n.$
\item There are constants $c,k$ in $\mathbb{N}$, $k\geq 2$, such that $F_{n,k}\leq \varphi_{\mathfrak{C}}(n)\leq n^c F_{n,k}$ for every $n$.
\item One has $\varphi_{\mathfrak{C}}(n)\geq 2^{n-1}$ for every $n$.
\end{enumerate}
\end{theorem}

In the cases (1) to (3) the generating function is rational. \emph {Albert and Atkinson} gave in 2005  examples of hereditary classes whose generating function is algebraic \cite{A-A}. In order to state their result, we recall
first that a power series
 $F(x):=\underset{n\geq 0}{\sum }a_{n}.x^{n}$ with $a_n$ in $\mathbb{C}$ is \emph{algebraic} if there exists a nonzero polynomial $Q(x,y)$ in $\mathbb{C}[x,y]$ such that $Q(x,F(x))=0$. The series $F$ is \emph{rational} if
$Q$ has degree 1 in $y$, that is, $F(x)=R(x)/S(x)$ for two polynomials in $\mathbb{C}[x]$ where $S\neq0$. Recall next that a permutation $\pi=a_1a_2\ldots a_n$ of $[n]$ is \emph{simple} if no proper interval of
$[n]$ ($\neq [n], \varnothing$ or $\{x \}$) is transformed into an interval. In other words, $\{a_i,a_{i+1},\ldots,a_j \}$ is not an interval in $[n]$ for every $1\leq i<j \leq n$, and either $i\not =1$ or $j\not =n$.  If $n\leq 2$ all permutations are simple and called trivial. Albert and Atkinson's theorem is the following:

\begin{theorem}
 If $\mathfrak{C}$ is an hereditary class of permutations containing only finitely many simple permutations, then the generating series of $\mathfrak{C}$, namely
$\sum_{n\geqq 0}\varphi _{\mathfrak{C}}(n)x^{n}$ is algebraic.
\end{theorem}

As an illustration of this result, let us mention that the class of permutations not above $2413$ and $3142$ contain no non trivial simple permutation
(these permutations are called \emph{separable permutations}).
The generating series of this class is
$\dfrac{1-x-\sqrt{1-6x+x^2}}{2}$ (see \cite{A-A-V}).

The simple permutations of small degree are
$1, 12, 21, 2413, 3142.$ Let  $S_n$ be the number of simple permutations of $[n]$. The values of $S_n$ for $n=1$ to $7$ are: $1,2,0,2,6,46,338$ \cite{Sloane}.
Asymptotically, $S_n$ goes to $\dfrac{n!}{e^2}$, a result obtained independently in
\cite{nosaki, A-A-K}.

\subsection{Permutations and bichains}\label{permutation-bichain}
Let $\sigma$ be a permutation of $[n]$. To $\sigma$ we associate the \emph{bichain} $C_{\sigma}:= ([n], \leq, \leq_{\sigma})$ where $\leq$ is the natural order on $[n]$
and $\leq_{\sigma}$ the linear order defined by
$i\leq_{\sigma} j$
if and only if  $\sigma(i) \leq \sigma(j)$.

%For example, let $\sigma$ be the permutation of $\underline 10$ given by  the sequence of its values: $5 (10) 4 9 3 82716$. The sequence of elements of  $\underline 10$ ordered  according to  $\leq_{\sigma}$ is: $9<_{\sigma}7<_{\sigma}5<_{\sigma}3<_{\sigma}1<_{\sigma}10<_{\sigma}8<_{\sigma}6<_{\sigma}4<_{\sigma}2$. Hence, this is the sequence of values of $\sigma^{-1}$.

For example, let $\sigma$ be the permutation of $\underline {10}$ given by  the sequence of its values: $2  4 6 8  (10)13579$. The sequence of elements of  $\underline {10}$ ordered  according to  $\leq_{\sigma}$ is: $6<_{\sigma}1<_{\sigma}7<_{\sigma}2<_{\sigma}8<_{\sigma}3<_{\sigma}9<_{\sigma}4<_{\sigma}10<_{\sigma}5$. Hence, this is the sequence of values of $\sigma^{-1}$, the inverse of $\sigma$.
Let us  represent $\sigma$ by its graph  in the product $\underline n\times \underline n$, that is the set $G(\sigma):=\{(i,\sigma(i)): i\in \underline n\}$  and  order this set componentwise, that is set $(i,\sigma(i))\leq (j, \sigma(j))$ if $i\leq j$ and $\sigma(i)\leq \sigma (j)$. Since $\sigma$ is bijective, the poset $G(f)$ is the intersection of two linear orders,  given respectively by the natural order on the first and on the second coordinate. If we identify each $i$ to $(i, \sigma (i))$, the order induced on $\underline n$ is the intersection of $\leq$ and $\leq_{\sigma}$. See Figure \ref{critique}.

\begin{figure}[h]
\begin{center}
\includegraphics[width=3in]{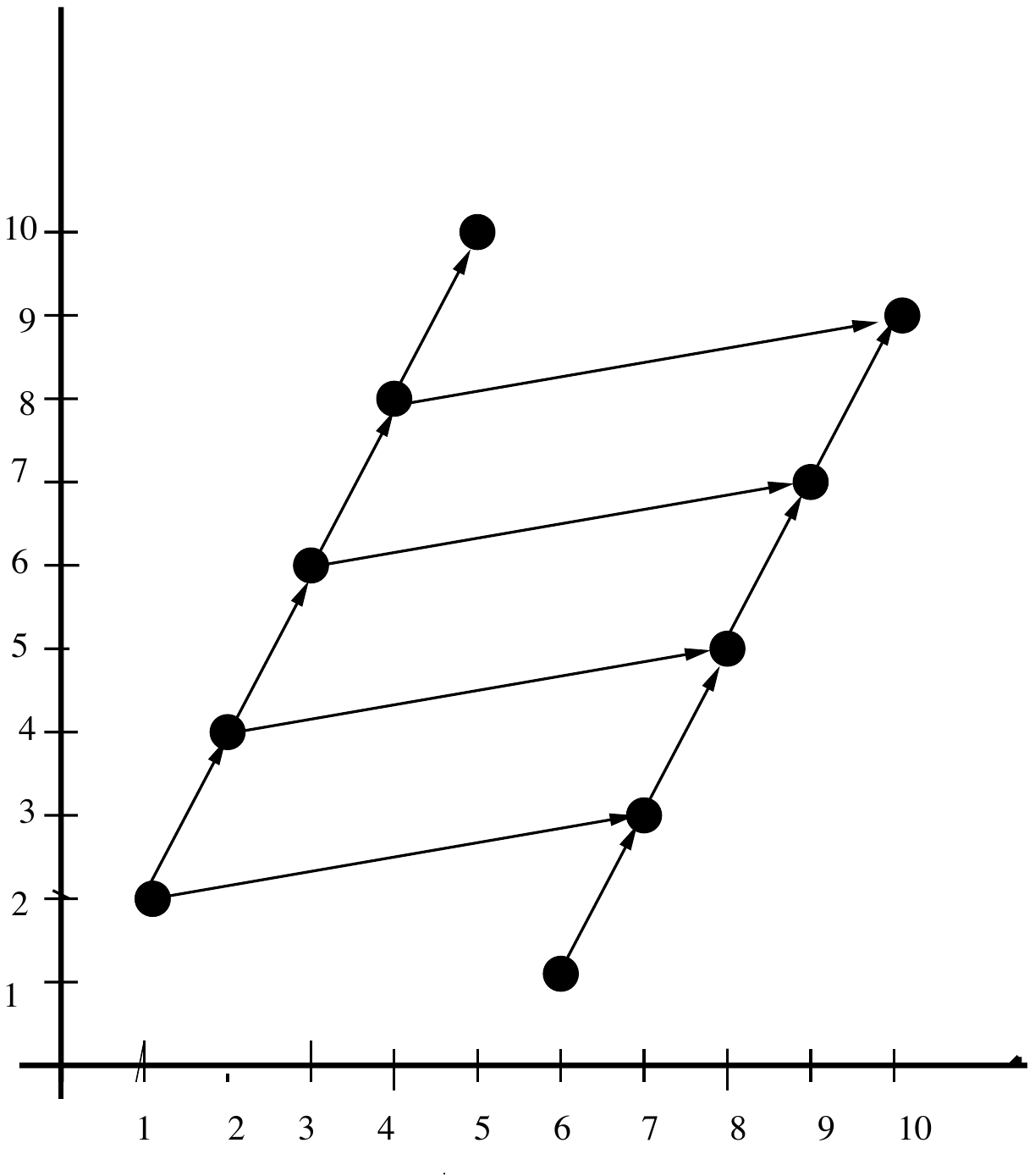}
\end{center}
\caption{Representation of a permutation of ten elements.}
\label{critique}
\end{figure}

\begin{lemma}\begin{enumerate}
\item If $B:=(E, L_1,L_2)$ is a finite bichain then $B$ is isomorphic to a bichain $C_{\sigma}$ for a unique permutation $\sigma$ on $[\vert E\vert~].$
\item If $\sigma$ and $\pi$ are two permutations then $\sigma \leq \pi$ if and only if $C_{\sigma}\leq C_{\pi}.$
\end{enumerate}
\end{lemma}

The correspondence between permutations and bichains was noted by Cameron \cite{cameron} (who rather associated to  $\sigma$ the pair $(\leq, \leq_{\sigma^{-1}}))$. It allows to study classes of permutations by means of the theory of relations.
In particular, via this correspondence, hereditary classes of permutations correspond to hereditary classes of bichains and, as we will see below,  simple permutations correspond to indecomposable bichains.

\section{Indecomposability and lexicographic sum}\label{sect:indecomposable}
Let  $\mathcal{R}:=(E,(\rho_i)_{i\in I})$ be a binary structure. A subset $A$ of $E$ is an  \textit{interval} of $\mathcal{R}$ if for each $i\in I$:
 $$(x\rho_i a \Leftrightarrow x\rho_i a')  \; \text{and} \; (a\rho_i x \Leftrightarrow a'\rho_i x) \; \text{for all} \; a,a'\in A \;\text{and}\;
x\notin A.$$
The empty set, the singletons and the whole set $E$ are intervals and said \textit{trivial}.
If $\mathcal{R}$ has no non trivial interval it is \textit{indecomposable}.
 For example, if $\mathcal{R}:=(E,\leq)$ is a chain, its intervals are the ordinary intervals. If $\mathcal{R}:=(E,\leq,\leq')$ is a bichain then
$A$ is an interval of $\mathcal{R}$ if and only if $A$
is an interval of $(E,\leq)$ and $(E,\leq')$.
Hence:

\begin{fact}
 A permutation $\sigma$ is simple if and only if the bichain $C_{\sigma}$ is indecomposable.
\end{fact}

The notion of indecomposability is rather old. The notion of interval goes back to Fra\"{\i}ss\'e \cite{fraisse2}, see also \cite{fraisse3}. A fundamental
decomposition result of a binary structures into intervals was
obtained by Gallai \cite{gallai}
(see \cite{ehren}
for further extensions). Hence, it is not surprising that several results on simple permutations were already known (for example their asymptotic
evaluation). Albert and Atkinson result recasted in terms of relational structures asserts that if $\mathfrak C$ \emph{is an hereditary class of finite bichains containing only
finitely many indecomposable bichains then the generating
 series of $\mathfrak C$ is algebraic}. The paper \cite{laflamme} contains several examples of infinite bichains $\mathcal B$ whose infinitely many members of  $Age (\mathcal B)$ are indecomposable.  \\
We will establish an extension to ordered binary structures in Theorem \ref{theo:algebraic}.\medskip

\noindent In the sequel, we recall the facts we need on lexicographic sums and the links with the indecomposability notion. Some are old (the notion of lexicographic sum goes back to Cantor).
\medskip

\noindent Let $\mathcal{R}:=(E,(\rho_i)_{i\in I})$ be a binary structure and $\mathfrak{F}:=(\mathcal{S}_x)_{x\in E}$ be a family of binary structures
$\mathcal{S}_x:=(E_x,({\rho_i}^x)_{i\in I})$, indexed by the
elements of $E$.  We  suppose that $E$ and the $E_x$ are non-empty.
The \emph{lexicographic sum} of $\mathfrak{F}$ over $\mathcal{R}$, denoted by $\underset{x\in \mathcal{R}}{\oplus }\mathcal{S}_x$, is the binary structure
$\mathcal{T}$ obtained by replacing each element
 $x\in E$ by the structure $\mathcal{S}_x$. More precisely, $\mathcal{T}=(Z,(\tau_i)_{i\in I})$ where $Z:=\{(x,y):x\in E,y\in E_x\}$ and for each $i\in I$,
$(x,y)\tau_i (x',y')$  if either $x\neq x'$ and $x\rho_i x'$ or $x=x'$ and $y{\rho_i}^x y'$.
\medskip

\noindent Trivially, if we replace each $\mathcal{S}_x$ by an isomorphic binary structure $\mathcal{S'}_x$, then $\underset{x\in \mathcal{R}}{\oplus }\mathcal{S'}_x$ is isomorphic to
 $\underset{x\in \mathcal{R}}{\oplus }\mathcal{S}_x.$
 Hence, we may suppose that the domains of the $\mathcal{S}_x$'s are pairwise disjoint. In this case we may slightly modify the definition above, setting
$Z:=\underset{x\in E}{\cup}{E_x}$ and for two elements
 $z\in E_x$ and $z'\in E_{x'}$,
$$z{\tau_i} z' \; \text{if either}\; x\neq x' \;\text{and}\; x\rho_i x' \; \text{or}\; x=x'\; \text{and}\; z{\rho_i}^x z'.$$

\noindent With this definition, each set $E_x$ is an interval of the sum $\mathcal T.$ Furthermore, let $Z/_{\equiv}$ be the quotient of $Z$ made of blocks of this partition
into intervals, let $p:Z\rightarrow Z/_{\equiv}$
be the natural projection and let $\mathcal R'$ be the image of $\mathcal T$ (that is $\mathcal R'=(Z/_{\equiv},(\rho'_i)_{i\in I})$ where
$\rho'_i=\{(p(x_1),p(x_2)):(x_1,x_2)\in \rho_i \}$. If we identify each block $E_x$ to the element $x,$ then $\mathcal R$ and $\mathcal R'$ coincide on pairs of distinct elements. They coincide if we
consider only reflexive relations.
Conversely, if  $\mathcal{T}:=(Z,(\tau_i)_{i\in I})$ is a binary structure and $(E_x)_{x\in E}$ is a partition of $Z$ into non empty intervals of
$\mathcal T,$ then $\mathcal T$ is the lexicographic sum of
$(\mathcal T \restriction_{E_x})_{x\in E}$ over the quotient $Z/_{\equiv}.$
In simpler words:

\begin{fact}
 The decompositions of a binary structure into lexicographic sums are in correspondence with the partitions of its domain into intervals.
\end{fact}

\noindent An important property of these decompositions is the following:

\begin{fact}
 The set of partitions of $E$ into intervals of $\mathcal R$, once ordered by refinement, is a sublattice of the set of partitions of $E$.
\end{fact}

\noindent Let us illustrate. Let us say that a lexicographic sum $\underset{x\in \mathcal{R}}{\oplus }\mathcal{S}_x$ is \emph{trivial} if $\vert E \vert=1$ or $\vert E_x \vert=1$ for all $x\in E$, otherwise it is
\emph{non trivial}; also a binary structure is \emph{sum-indecomposable} if it can not be isomorphic to a non trivial lexicographic sum. We have immediately:

\begin{fact}
 A binary structure is sum-indecomposable if and only if it is indecomposable.
\end{fact}

\begin{proposition}\label{pro:sumlex}
Let $\mathcal{R}$ be a finite binary structure with at least two elements. Then $\mathcal{R}$ is isomorphic to a lexicographic sum $\underset{x\in \mathcal{S}}
{\oplus }\mathcal{R}_x$ where $\mathcal{S}$ is
indecomposable with at least two elements. Moreover, when $\mathcal S$ has at least three elements, the partition of $\mathcal R$ into intervals is unique.
\end{proposition}

If the set $\mathcal S$ in Proposition \ref{pro:sumlex} has two elements, then the decomposition is not necessary unique, a fact which leads to the notion of strong interval.
We recall that an interval $A$ of  a binary structure $\mathcal{R}$ is \emph{strong} if it is non-empty and overlaps no other interval, meaning that if $B$ is an interval such that
$A\cap B\neq \emptyset$ then either $B \subseteq A$ or
$A \subseteq B$. We say that $A$ is \emph{maximal} if it is maximal for inclusion among strong intervals distinct from the domain $E$ of $\mathcal R$. The maximal strong interval form a partition of $E$, provided that some maximal exists; in this case $E$ is  \emph{non-limit} \cite{Ke} or, equivalently,   \emph{robust} \cite {c-d}.  Evidently, this partition exists whenever $E$ is finite. The reader will easily check that when this partition exists and the quotient is indecomposable then every other non-trivial partition into intervals is finer. Hence,  in Proposition \ref{pro:sumlex} above, if $\mathcal S$  has at least three elements,  the intervals in the decomposition are strong and thus the decomposition is unique.

\noindent Let us say that $\mathcal R:=(E,(\rho_i)_{i\in I})$ is \emph{chainable} if there is a linear order, $\leq$, on $E$ such that, for each $i$,
$x\rho_i y\Leftrightarrow x'\rho_i y'$ for every $x,y, x',y'$ such that
$x\leq y\Leftrightarrow x'\leq y'$. If $\mathcal R$ is reflexive, this amounts to the fact that each $\rho_i$ is either the equality relation
$\triangle_E$, the complete relation $E\times E$ or a linear order;
moreover if $\rho_i$ and $\rho_j$ are two linear orders, they coincide or are opposite. Note that if furthermore $\mathcal R$ is ordered then the $\rho_i$' which are linearly ordered are equal or opposite to the given order.

We arrive to the fundamental  decomposition theorem of Gallai \cite{gallai} (see for example \cite{Ke}, \cite {ehren}, \cite{c-d} for extensions to  infinite structures)
\medskip

\begin{theorem}\label{theo:decomposition}
 Let $\mathcal R$ be a finite binary structure with at least two elements, then $\mathcal R$ is a lexicographic sum $\underset{x\in \mathcal{S}}{\oplus}
{\mathcal R}_x$ where $\mathcal S$ is either indecomposable
with at least three elements or a chainable binary structure with at least two elements and the $V(\mathcal R_x)$'s are strong maximal intervals of $\mathcal R$.
\end{theorem}

\noindent  For our purpose, we need to introduce the following notion.

\noindent Let $\tau$ be an ordered structure with two elements. An ordered structure $\mathcal S$ is said $\tau$-\emph{indecomposable} if it cannot be decomposed into a lexicographic sum indexed by $\tau.$
If $\mathfrak A$ is a class of structures, we denote by $\mathfrak A(\tau)$ the set of members of $\mathfrak A$ which are $\tau$-indecomposable.

\begin{lemma}\label{lem:decomposition}
Let $\mathcal S:=(\{0,1\},\leq,(\rho_i)_{i\in J})$,
 with $0<1$, be an ordered structure with two elements. If $\mathcal R$ is a lexicographic sum $\underset{x\in \mathcal{S}}{\oplus }\mathcal{R}_x$ and $\mathcal R_0$ is $\mathcal S$-indecomposable,
then the partition $\mathcal R_0,\mathcal R_1$ is unique.
\end{lemma}

\noindent We extend the notion of lexicographic sum to collections of non-empty binary structures. Given a non-empty binary structure $\mathcal{R}$ and classes  $\mathfrak{A}_x$ of non-empty binary
structures for each $x\in \mathcal{R}$, let us denote by
 $\underset{x\in \mathcal{R}}{ \oplus}\mathfrak{A}_x$ the class of all binary structures of the form
$\underset{x\in \mathcal{R}}{\oplus }\mathcal{S}_x$ with $\mathcal{S}_x \in \mathfrak{A}_x$. If $\mathfrak{A}_x=\mathfrak A$ for every $x\in V(\mathcal R)$ we denote this class by  $\underset{ \mathcal{R}}{ \oplus}\mathfrak{A}$.  If  $\mathcal R:=(\{0,1\},\leq,(\rho_i)_{i\in J})$,
 with $0<1$, $\mathfrak{A}_0:=\mathfrak A$  and  $\mathfrak{A}_1 := \mathfrak B$, we set   $\underset{x\in \mathcal{R}}{ \oplus}\mathfrak{A}_x=\mathfrak A\underset{\mathcal R}{\oplus}\mathfrak {B}$.
Also if $\mathfrak A$ and $\mathfrak B$ are two classes of binary structures, we set
$\underset{\mathfrak{A}}{\oplus}\mathfrak{B}:=\{\underset{x\in \mathcal{R}}{\oplus }\mathcal{S}_x:~~\mathcal{R}\in \mathfrak{A},~~\mathcal{S}_x \in \mathfrak{B} \; \text{for each} \; x\in V(\mathcal{R}) \}$.

We say that a collection  $\mathfrak C$ of binary structures is \emph{sum-closed} if $\underset{\mathfrak{C}}{\oplus}{\mathfrak{C}}\subseteq {\mathfrak{C}}$.
The \emph{sum-closure} $cl(\mathfrak{C})$ of $\mathfrak{C}$ is the
smallest sum-closed set that contains $\mathfrak{C}$. If we define $\mathfrak C_0:=\mathfrak C$ and $\mathfrak C_{n+1}:=\underset{\mathfrak{C}}{\oplus}
\mathfrak C_n$, then
$cl(\mathfrak{C})=\underset{n=1}{\overset{\infty }{\cup }} \mathfrak{C}_{n}$.  If $\mathfrak{C}$ is a class of bichains, $cl(\mathfrak{C})$ is also a class of bichains and the class of corresponding permutations is said  \emph{wreath-closed} \cite{A-A}.  If $\mathfrak C$ is made of reflexive structures and contains a one element  structure, say $\bold 1$, then $cl(\mathfrak C)=\underset{cl(\mathfrak{C})}{\oplus}{cl(\mathfrak{C})}$. If in addition $\mathfrak C$ contains the empty structure then $cl(\mathfrak C)$ is hereditary.
%% need to know the one-element members of $\mathfrak C$ and some of its indecomposable members in order to obtain the closure $\mathfrak C$.
%%
%Let $\mathfrak D$ be a class of indecomposable structures, each with at least two elements.
%In the sequel we suppose each member $S\in \mathfrak D$ made of reflexive binary relations.
%Let $\mathfrak A$ be a collection of finite structures. Set   $\mathfrak C_0:=\mathfrak A$, $\mathfrak C_1:=\mathfrak D$ and inductively $\mathfrak C_{n+1}:=\underset{\mathfrak{D}}{\oplus}
%\mathfrak C_n$. Finally, set
%$cl_{\mathfrak A} (\mathfrak{D}):=\underset{n=0}{\overset{\infty }{\cup }} \mathfrak{C}_{n}$. If $\mathfrak A$ is the reflexive one-element structure, denote this class by $cl(\mathfrak{D})$.  \medskip
%
%
%

\noindent We denote by $Ind(\Omega_I)$ the collection of finite indecomposable members of $\Omega_I.$  If $\mathcal{R}$ is a binary structure, we denote by $Ind(\mathcal{R})$ the collection of its finite induced substructures
which are indecomposable. For example, if $\mathcal{R}$ is a cograph or a  serie-parallel poset then the members of $Ind(\mathcal{R})$ have at most two elements (a graph (undirected) is a \emph{cograph} if no induced subgraph is isomorphic to $P_4$, the path on $4$ vertices,  and a poset is  \emph{serie-parallel} if its comparability graph is a cograph). \\

\noindent Let $\mathfrak{D}$ be a hereditary class of $Ind(\Omega_I)$. Set $\sum \mathfrak{D}:= \{\mathcal{R} \in {\Omega_I} : Ind(\mathcal{R}) \subseteq
\mathfrak{D}\}$.

\begin{theorem}\label{theo:closed}
 If all members of $\mathfrak D$ are reflexive, then $\sum \mathfrak{D}=cl(\mathfrak{D}).$
\end{theorem}

\begin{proof}
Inclusion $\sum \mathfrak{D}\supseteq {cl(\mathfrak{D})}$ holds under assumption that all members of $\mathfrak{D}$ are reflexive. Conversely, if $\mathcal{R}\notin cl(\mathfrak{D})$ then either $\mathcal{R}$ is
indecomposable in which case $\mathcal{R}\notin \mathfrak{D}$ or $\mathcal{R}$ can not be expressed as a lexicographic sum of structures of $\mathfrak{D}$ hence $\mathcal{R}\notin \sum \mathfrak{D}.$
\end{proof}

In the sequel, we consider only ordered structures made of reflexive binary relations. Let $\Gamma _{d}$ be the subclass of reflexive members of  $\Theta_d$.

Let $\mathfrak A$ be a subclass of $\Gamma _{d}$; for $i\in \NN$ let  $\mathfrak A_{(i)}$,
resp. $\mathfrak A_{(\geq i)}$, be the subclass made of its members   which have $i$ elements, resp. at least $i$ elements.

\begin{lemma}\label{lem:union}
Let $\mathfrak D$ be a  class made of non-empty indecomposable members of  $\Gamma _{d}$ such that $\mathfrak D_{(1)}$ is reduced to the one-element structure ${\bf 1}$. Let $\mathfrak A$ be  the sum-closure of $\mathfrak D$ and  for each $\mathcal S\in  \mathfrak D_{(2)}$, let $\mathfrak A (\mathcal S)$ be the subclass of $\mathcal S$-indecomposable members of $\mathfrak A$. Set $\mathfrak A_{\mathcal S}:= \underset{ \mathcal{S}}{ \oplus}\mathfrak{A}$ if $\mathcal S\in  \mathfrak D_{( \geq3)}$ and otherwise set $\mathfrak A_{\mathcal S}:=\mathfrak A(\mathcal S)\underset{\mathcal S}{\oplus}\mathfrak {A}$ if $\mathcal S\in  \mathfrak D_{(2)}$ and  $\mathcal S:=(\{0,1\},\leq,(\rho_i)_{i\in J})$ with $0<1$. Then:

\begin{equation} \label{eq:classe}
\mathfrak A=\{\bf 1\}\cup \underset{\mathcal S\in \mathfrak D_{(\geq 2)}}{\bigcup}\mathfrak A_{\mathcal S}\end{equation}
and
\begin{equation}\label{eq:classeindec}
\mathfrak A(\mathcal S)=\mathfrak A \setminus \mathfrak A_{\mathcal S}
\end{equation}
for every $\mathcal S \in \mathfrak D_{(2)}$.
%where $\mathfrak A_x=\mathfrak A$ for all $x\in \mathcal R$.

%\mathfrak A(\mathcal S)=\{\bf1\}\cup \bigcup
%\underset{\mathcal R\in \mathfrak D_{(\geq 2)}}
%{\bigcup}\underset{x\in \mathcal R}{\oplus} {\mathfrak A_x}, \; \text{for all}\; \mathcal S\in\mathfrak D_{(2)}\end{equation}
%where $\mathfrak A_x=\mathfrak A$ for all $x\in \mathcal R$.
Furthermore, all sets in  equation (\ref {eq:classe} ) are pairwise disjoint.
\end{lemma}

\begin{proof}
Let's denote by $(1)$ (respectively by $(2)$) the left-hand side (respectively the right-hand side) of Equation \ref{eq:classe}. Inclusion $(2)\subseteq (1)$ is obvious because $\mathfrak A$
is sum-closed according to Theorem \ref{theo:closed}.
To prove inclusion $(1)\subseteq (2)$, let $\mathcal R$ be in $(1)$, if $\mathcal R$ has one element then it is in $(2)$, otherwise, according to Theorem \ref{theo:decomposition}, $\mathcal R$ is
a lexicographic sum $\underset{x\in \mathcal S}\oplus \mathcal R_x$ where $\mathcal S$ is either indecomposable with at least three elements or a chainable binary structure with at least two
elements and each   $\mathcal R_x\in \mathfrak A$ is a strong interval of $\mathcal R$ for each $x\in \mathcal S$.  In the first case, $\mathcal R$ is in
$\underset{\mathcal S\in \mathfrak{D_{(\geq  3)}}} \bigcup \underset{x\in \mathcal S}\oplus \mathfrak A$, hence in $(2).$
In the second case,  $\mathcal S$ is chainable with $n$ elements, $n\geq 2$,   and we may set $\mathcal S:=(\{0,1,\cdots,n-1\},\leq,(\rho_i)_{i\in I})$ with $0<1<\cdots<n-1$.
Set $\mathcal S':=\mathcal S \restriction_{\{0,1\}}$, $\mathcal S'':=\mathcal S \restriction_{\{1,\cdots,n-1\}}$, $\mathcal R'_0:=\mathcal R_0$ and
$\mathcal R'_1:= \underset{x\in \mathcal S''}\oplus \mathcal R_x$. We have obviously  $\mathcal R= \mathcal R'_0 \underset{\mathcal S'}\oplus \mathcal R'_1$.
Since $\mathcal R_0$  is a strong interval of $\mathcal R$, $\mathcal R'_0$ is  $\mathcal S'$-indecomposable, hence $\mathcal R$ belongs to
$\underset{\mathcal S'\in \mathfrak D_{(2)}}{\bigcup}(\mathfrak A(\mathcal S')\underset{\mathcal S'}{\oplus}\mathfrak A)$ which is a subset of $(2).$
The fact that these sets are pairwise disjoint follows from Proposition \ref{pro:sumlex} and Lemma \ref{lem:decomposition}.

%
%As every interval of the linear order $\leq$ is an interval of $\mathcal S$, then for every integer $1\leq k\leq n-1$ one has $\mathcal S={\mathcal{S}_1}^k \underset{\mathcal S'}\oplus {\mathcal{S}_2}^k$, where $\mathcal S'=\mathcal S \restriction_{\{0,1\}}$ and ${\mathcal{S}_1}^k, {\mathcal {S}_2}^k$ are chainable with $k$ and $n-k$ elements respectively which embed in $\mathcal S.$.
%Thus $\mathcal R=\underset{x\in \mathcal S}\oplus \mathcal R_x=\mathcal R_0\underset{\mathcal S}\oplus \cdots \underset{\mathcal S}\oplus \mathcal R_{n-1}=\mathcal R'_0 \underset{\mathcal S'}\oplus \mathcal R'_1,$ where
%$\mathcal R'_0=\mathcal R_0\underset{{\mathcal{S}_1}^k}\oplus \cdots \underset{{\mathcal {S}_1}^k}\oplus \mathcal R_{k-1}$ and $\mathcal R'_1=\mathcal R_k\underset{{\mathcal {S}_2}^k}\oplus \cdots \underset{{\mathcal {S}_2}^k}\oplus \mathcal R_{n-1}.$ If we impose  $k=1$ then, there is a unique way to decompose $\mathcal R$ into intervals provided that we take $\mathcal R_0$ in $\mathfrak A(\mathcal S').$ Thus, $\mathcal R$ is in
%$\underset{\mathcal S\in \mathfrak L}{\bigcup}(\mathfrak A(\mathcal S)_0\underset{\mathcal S}{\oplus}\mathfrak A_1)$ which is a subset of $(2).$
 Equality \ref{eq:classeindec} is obvious: the $\mathcal S$-indecomposable members of $\mathfrak A$ are those which cannot be writen as $\mathcal S$-sums.
 \end{proof}

\bigskip

%Let $\{\mathcal S_1,\cdots,\mathcal S_p\}$  be an enumeration of $\mathfrak D_{(2)}$.

In the sequel,  we count. Our structures being ordered  we may choose a unique representative of an $n$-element structure on  the set $\{0, \dots, n-1\}$,  the ordering being the natural order.
 Let $\mathcal H$ and $\mathcal K$  be  the generating series of
$\mathfrak{A}$ and $\mathfrak{D}_{(\geq 3)}$ and let $\mathcal K(\mathcal H)$ be the series obtained by substituting the indeterminate $x$ by $\mathcal H$.  Let  $\mathcal H_{\mathfrak {A}(\mathcal S)}$ and  $\mathfrak {A}_{\mathcal S}$ be the generating series of $\mathfrak A(\mathcal S)$ and $\mathfrak A_{\mathcal S}$ for $\mathcal S\in \mathfrak{D}_{(\geq 2)}$. And let $p$ be the cardinality of $\mathfrak{D}_{(2)}$.

\begin{lemma}\label{lem:hereditary-series}
\begin{equation}\label{eq1}
 (p-1)\mathcal {H}^2+(x-1+\mathcal K(\mathcal{H}))\mathcal {H}+x+\mathcal K(\mathcal{H})=0.
 \end{equation}

  \begin{equation}\label{eq2}
 \mathcal{H}_{\mathfrak {A}(\mathcal S)}=\frac{\mathcal{H}}{1+\mathcal {H}} \;\text{for every}\;  \mathcal S\in \mathfrak{D}_{(2)}.
\end{equation}

%\mathcal H_i= \frac{x+\mathcal K(\mathcal{H})} {1-(p-1)\mathcal H}\; ~~\text{for every}\;  i\in \{1,\cdots,p\}.
\end{lemma}

\begin{proof}

\noindent Let us prove that Equation \ref{eq2} holds. Let $\mathcal S\in \mathfrak{D}_{( 2)}$. Since by  definition in Lemma \ref{lem:union},  $\mathfrak A_{\mathcal S}=\mathfrak A(\mathcal S)\underset{\mathcal S}{\oplus}\mathfrak {A}$, we have $\mathcal H_{\mathfrak A_{\mathcal S}}=\mathcal H_{\mathfrak {A}(\mathcal S)}.\mathcal H$. From Equation (\ref{eq:classeindec}),  we deduce  $\mathcal H_{\mathfrak A_{\mathcal S}}= \mathcal H-\mathcal H_{\mathfrak A_{\mathcal S}}.\mathcal H$. Since the coefficients of $\mathcal H$ are non-negative, the  series $1+ \mathcal H$ is invertible, hence $\mathcal H_{\mathfrak A_{\mathcal S}}=\frac{\mathcal H}{1+\mathcal H}$ as claimed in Equation (\ref {eq2}).

 \noindent Let us prove that Equation \ref{eq1} holds.
Let $\mathcal S\in \mathfrak{D}_{(n)}$, with $n\geq 3$. Since by  definition in Lemma \ref{lem:union},  $\mathfrak A_{\mathcal S}= \underset{ \mathcal{S}}{ \oplus}\mathfrak{A}$,  we have  $\mathcal H_{\mathfrak A_{\mathcal S}}=\mathcal H^n$. From this, we deduce that the generating series of $\underset{\mathcal S\in \mathfrak D_{(\geq 3)}}{\bigcup}\mathfrak A_{\mathcal S}$ is equal to $\mathcal K(\mathcal{H})$.
From Equation (\ref {eq2}), we deduce that the generating series of $\mathcal H_{\mathfrak A_{\mathcal S}}$ is $\frac{\mathcal{H}^2}{1+\mathcal {H}}$. Hence the generating series of  $\underset{\mathcal S\in \mathfrak D_{( 2)}}{\bigcup}\mathfrak A_{\mathcal S}$ is equal to $p\frac{\mathcal{H}^2}{1+\mathcal {H}}$.

% From Lemma \ref{lem:union} we have, passing to generating functions:
%
%\begin{IEEEeqnarray}{rCl}
%\label{eq}
%\underset{n\geq 1}\sum \varphi_{\mathfrak A}(n)x^n&=&x+\underset{i=1}{\overset{p}{\sum }}\underset{n\geq 2}{\sum }\left[ \overset%
%{n-1}{\underset{l=1}{\sum }}\varphi_{\mathfrak A(\mathcal S_i)}(l).\varphi_{\mathfrak A}(n-l)\right]x^n \nonumber
%\\
% && +\: \underset{n\geq 3}{\sum }\left[\overset{n}{\underset{l=3}{\sum }}\;
% \underset{n_1+\cdots+n_l=n}{\sum }\varphi_{\mathfrak D_{(\geq 3)}}(l)\varphi_{\mathfrak A}(n_1)\cdots \varphi_{\mathfrak A}(n_l)\right]x^n.
%\end{IEEEeqnarray}
%
%Using the notations above,  we obtain:

Substituting these values  in Equation (\ref{eq:classe}), we obtain
\begin{equation}\label{eq:4}\mathcal H=x+p\frac{\mathcal H^2}{1+\mathcal H}+\mathcal K(\mathcal{H}).
\end{equation}

A straightforward computation yields Equation (\ref{eq1}).

%Replacing this value in Equation (\ref{eq:3}), we obtain Equation (\ref{eq1}).
%From this we get:
%\begin{equation}\label{eq:3'}\mathcal H.\underset{i=1}{\overset{p}{\sum }}\mathcal H_i.=\mathcal H-x-\mathcal K(\mathcal{H}).
%\end{equation}
%
%In the same way:
%\begin{equation}\label{eq:4}\mathcal H_i=x+\underset{j=1,j\neq i}{\overset{p}{\sum }}\mathcal H_j.\mathcal H+\mathcal K(\mathcal{H})~~ \;\text{for}\; i\in \{1,\cdots,p\}.
%\end{equation}
%
%
%
%By eliminating the $\mathcal H_i$, $i\in \{1,\cdots, p\}$  from Equation \ref{eq:3}, we obtain:
%\begin{equation}\label{eq:5}
%(p-1)\mathcal {H}^2+(x-1+\mathcal K(\mathcal{H})){\mathcal {H}}+x+\mathcal K(\mathcal{H})=0.
%\end{equation}
%The $p$ equations given by Equation \ref{eq:4} form a linear system whose unknowns are the $\mathcal H_i$. A Gaussian elimination yields the result.
\end{proof}

Let us say that a class of finite structures is \emph{algebraic} if its generating series is algebraic.

\begin{corollary}\label{cor:wqoalgebraic}
Let $\mathfrak D$ be a  class made of non-empty indecomposable members of  $\Gamma _{d}$ such that $\mathfrak D_{(1)}$ is reduced to the one-element structure ${\bf 1}$. If  $\mathfrak D$ is algebraic then its  sum-closure and the subclass $\mathfrak A_{\mathcal S}$ consisting of the  $\mathcal S$-indecomposable members of the  sum-closure $\mathfrak A$ are algebraic for each $\mathcal S\in \mathfrak D_{(2)}$,  .
\end{corollary}

\section{Well-quasi-ordered hereditary classes} \label{sect:well-quasi-order}
Let $\mathfrak{C}$ be a subclass of $\overline \Omega_\mu$ and $\mathcal{A}$ be a poset. Set $\mathfrak{C}.\mathcal{A}:=\{(\mathcal{R},f):\mathcal{R}\in \mathfrak{C},
f:V(\mathcal{R})\rightarrow \mathcal{A}\}$ and
$(\mathcal{R},f)\leq (\mathcal{R'},f')$ if there is an embedding $h:\mathcal{R}\rightarrow \mathcal{R'}$ such that $f(x)\leq f'(h(x))$ for all $x\in
V(\mathcal{R})$.\\ We recall that $\mathcal{A}$ is
\textit{well-quasi ordered (wqo)} if $\mathcal{A}$ contains no infinite antichain and no infinite descending chain.
We say that $\mathfrak{C}$ is
\textit{hereditary wqo} if $\mathfrak{C}.\mathcal{A}$
is \textit{wqo} for every wqo  $\mathcal{A}$. It is
clear that every class  which is hereditary wqo  is wqo. If $\mathfrak{C}$ is reduced to a single structure $\mathcal{R}$, it is hereditary wqo provided that $\mathcal R$ is finite (this follows from the fact that if $\mathcal A$ is wqo then its power $\mathcal A^n$ ordered coordinatewise is wqo for each integer $n$).  If $\mathcal R$ is infinite, this does not hold. Also, a finite union of hereditary wqo  classes is hereditary  wqo; hence every finite subclass
$\mathfrak C$ of $\Omega_{\mu}$ is hereditary wqo.

A longstanding open question ask whether $\mathfrak{C}$ is hereditary wqo whenever the class $\mathfrak{C}.\underline 2$ of the elements of $\mathfrak{C}$ labelled by  $\underline 2$, the $2$-element antichain, is wqo.

If $Ch$ is the class of finite chains, $Ch.\mathcal A$ identifies to the set $\mathcal A^*$
of finite words over the alphabet $\mathcal A$ equipped with the Higman ordering. The fact that $Ch$ is hereditary wqo is a famous result due to Higman \cite{higman}. We also note the following fact:

\begin{fact}
If a subclass $\mathfrak C$ of $\Omega_\mu$ (with $I$ finite) is hereditary wqo, then $\downarrow\mathfrak C$, the least hereditary subclass of $\Omega_\mu$ containing $\mathfrak C$, is hereditary wqo.
\end{fact}

\noindent We recall the following result of
\cite{pouzet 72}.

\begin{theorem}\label{thm:bounds} If the signature is finite, a subclass of $\Omega_\mu$ which is hereditary and hereditary wqo has finitely many bounds.
\end{theorem}

\noindent Behavior of the profile of special hereditary classes, the \emph{ages} of Fra\"{i}ss\'e,  and the link with \textit{wqo} classes were considered by the second author in the early seventies (see \cite {pouzet.tr.1978} and \cite{pouzet} for a survey). The case
of graphs, tournaments and other combinatorial structures was elucidated
more recently (see the survey of \cite{klazar}).\\

\begin{proposition}\label{prop:wqo}
If a hereditary class $\mathfrak{D}$ of $Ind(\Omega_I)$ is hereditary \textit{wqo} then $\sum \mathfrak{D}$ is hereditary \textit{wqo} and $\sum \mathfrak{D}$ has finitely many bounds.
\end{proposition}

\begin{proof}
The second part of the proposition follows from Theorem  \ref{thm:bounds}  above. In our case of binary structures, we may note that  the proof is straightforward. The first part uses  properties of wqo posets, and follows from
Higman' theorem  on algebras preordered by divisibility (1952) \cite{higman}. Instead of recalling the result we give a direct proof.  Let $\mathcal A$ a poset which is wqo and consider
$(\sum \mathfrak D). \mathcal A.$ If $(\sum \mathfrak D). \mathcal A$ is not wqo, then according to one of preliminary result  of Higman, it contains some non finitely generated final segment ($\mathcal F$ is
 a final segment if $x\in \mathcal F$ and $x\leq y$ imply $y\in \mathcal F$). According to Zorn lemma, there is a maximal one, say $\mathcal F$, with respect to inclusion among final segments having this property.
 Let $\mathcal I:= (\sum \mathfrak D). \mathcal A \setminus \mathcal F$ be the complement of $\mathcal F$ in $(\sum \mathfrak D). \mathcal A.$ The set $\mathcal I$ is then wqo.
Let $\mathfrak R:=(\mathcal R_0,f_0),\cdots, (\mathcal R_n,f_n), \cdots$ be an infinite antichain of minimal elements of $\mathcal F.$ As $\mathfrak D. \mathcal A$ is wqo because $\mathfrak D$ is hereditary wqo,
we can suppose that no element of this antichain is in $\mathfrak D. \mathcal A.$
Then, according to Proposition \ref{pro:sumlex} and Theorem \ref{theo:decomposition}, for every $i\geq 0$ there exists an indecomposable structure $\mathcal S_i$ and  non-empty structures
$(\mathcal R_{ix})_{x\in V(S_i)}$ such that  $\mathcal R_i= \underset{x\in \mathcal S_i}\oplus \mathcal R_{ix}$.  Since $(\mathcal R_{ix},f_i\restriction_{V(\mathcal R_{ix})})$ strictly embeds into
$(\mathcal R_i,f_i)$ we have  $(\mathcal R_{ix},f_i\restriction_{V(\mathcal R_{ix})})\in \mathcal I$ for every $i\geq 0$ and $x\in \mathcal S_i$. Since $\mathcal I$ is wqo, and $\mathfrak D$ is hereditary wqo,
$\mathfrak D.\mathcal I$ is wqo, thus the infinite sequence  $(\mathcal S_0,g_0),\cdots,(\mathcal S_i,g_i),\cdots$ of $\mathfrak D.\mathcal I$,
where $g_i(x):=(\mathcal R_{ix},f_i\restriction_{V(\mathcal R_{ix})})$,  contains an increasing pair $(\mathcal S_p,g_p)\leq (\mathcal S_q,g_q)$ for some $p<q$. Which means that
there is an embedding $h:\mathcal S_p \rightarrow \mathcal S_q$ such that $g_p(x)\leq g_q(h(x))$ for all $x\in V(S_p)$, that is
$(\mathcal R_{px},f_p\restriction_{V(\mathcal R_{px})})\leq (\mathcal R_{qh(x)},f_q\restriction_{V(\mathcal R_{qh(x)})})$ for all $x\in S_p$. It follows that  $(\mathcal R_p,f_p)=
\underset{x\in \mathcal S_p}\oplus (\mathcal R_{px},f_p\restriction_{V(\mathcal R_{px})})\leq \underset{x\in \mathcal S_q}\oplus(\mathcal R_{qx},f_q\restriction_{V(\mathcal R_{qx})})=(\mathcal R_q,f_q)$
which  contradicts that $\mathfrak R$ is an antichain. Thus
$(\sum \mathfrak D). \mathcal A$ is  wqo and hence $(\sum \mathfrak D)$ is hereditary wqo.
\end{proof}
\bigskip

\noindent Proposition \ref{prop:wqo} particulary holds if $\mathfrak{D}$ is finite. If $\mathfrak{D}$ is the class $Ind_k(\Omega_I)$ of indecomposable structures of size at
most $k$ then according to a result of Schmerl and Trotter,
1993 (\cite{S-T}), the bounds of $\sum Ind_k(\Omega_I)$ have size at most $k+2$.  When $\mathfrak{D}$ is  made of bichains, Proposition \ref{prop:wqo} was  obtained
by Albert and Atkinson \cite {A-A}.\\
An immediate corollary is:

\begin{corollary}\label{cor:wqo}
If a hereditary class of $\Omega_I$ contains only finitely many indecomposable members then  it is wqo and has finitely many bounds. \end{corollary}

\noindent We say that a class $\mathfrak{C}$ of relational structures is \emph{hereditary rational}, resp. \emph{hereditary algebraic} if the generating function of every hereditary subclass of $\mathfrak{C}$
is rational, resp. algebraic.
Albert, Atkinson and Vatter \cite{A-A-V} proved that  hereditary rational classes of permutations are  wqo. This fact can be extended to hereditary  algebraic classes.

\begin{lemma} \label{lem:wqo}
A hereditary class $\mathfrak{C}$ which is hereditary algebraic is wqo.
\end{lemma}

\begin{proof} If  $\mathfrak{C}$ contains an infinite antichain, there are uncountably many hereditary subclasses of $\mathfrak{C}$ and in fact an uncountable chain of subclasses; these classes
 provides uncountably many generating series.  Some of these series cannot be algebraic. Indeed, according to C. Retenauer \cite{C-R},  a generating series with rational coefficients which is algebraic over
 $\mathbb C$ is algebraic over $\mathbb Q$. Since the generating series we consider have integer coefficients, there are  algebraic over $\mathbb Q$, hence there are only countably many  such series.
\end{proof}

If  $\mathfrak C$ and $\mathfrak D$ are two hereditary  classes, then the generating series satisfy the identity $\mathcal H_{\mathfrak C\cup \mathfrak D}= \mathcal H_{\mathfrak C}+\mathcal H_{\mathfrak D}-\mathcal H_{\mathfrak C\cap \mathfrak D}$. From this simple equality we have:

\begin{lemma} \label{lem:algebraic} The union of two hereditary rational (resp. algebraic) classes is hereditary rational (resp. algebraic).
\end{lemma}

\begin{corollary} A minimal non-hereditary rational  or a minimal non-hereditary algebraic class $\mathfrak C$ is the age of some relational structure.
\end{corollary}
\begin{proof} According to Lemma \ref{lem:algebraic}, $\mathfrak C$ cannot be the union of two proper hereditary subclasses, hence this is an ideal, thus an age.
\end{proof}

\begin{theorem}\label{theo:algebraic}
 Let $d$ be an integer. If  an  hereditary class $\mathfrak C$ of $\Gamma _{d}$ contains only finitely many  indecomposable members then it is  algebraic.
\end{theorem}

\noindent We follows essentially  the lines of Albert-Atkinson proof.
We  do an inductive proof over the hereditary subclasses of $\mathfrak C$. But for that,  we need to prove more, namely that $\mathfrak C$ and each $\mathfrak C(\mathcal S)$ for $\mathcal S\in  Ind(\mathfrak C)_{(2)}$,  are algebraic (this is the only difference with Albert-Atkinson proof). To avoid unessential complications,  we take out the empty relational structure of $\Gamma_d$, that is we suppose that $\mathfrak C$ is made of non-empty structures. Let   $\mathfrak A:=\sum Ind(\mathfrak{C})$. If $\mathfrak C= \mathfrak A$ then by Corollary \ref {cor:wqoalgebraic},  $\mathfrak C$ and each $\mathfrak C(\mathcal S)$ for $\mathcal S\in  Ind(\mathfrak C)_{(2)}$, are algebraic.  Thus the result is proved. If $\mathfrak C\not = \mathfrak A$, we may suppose that for each proper hereditary subclass $\mathfrak C'$ of $\mathfrak C$, both $\mathfrak C'$ and  $\mathfrak C'(\mathcal S)$ for each $\mathcal S\in  Ind(\mathfrak C')_{(2)}$, are algebraic. Indeed,  otherwise, since by  Corollary \ref{cor:wqo},  $\mathfrak C$ is wqo, it contains a minimal hereditary subclass not satisfying this property and we may replace $\mathfrak C$ by this subclass.
Let  $\mathcal S\in  Ind(\mathfrak C)_{(2)}$. Let  $\mathfrak C(\mathcal S)$ be the subclass of $\mathfrak C$ made of $\mathcal S$-indecomposable members of $\mathfrak C$. Let  $0$ and $1$, with $0<1$, be the two elements of $V(\mathcal S)$,  we  set
$\mathfrak C_{\mathcal S}:= (\mathfrak C(\mathcal S)\underset {\mathcal S }\oplus\mathfrak C)\cap \mathfrak C$. Let $\mathcal S \in Ind(\mathfrak C)_{(\geq 3)}$ we   set
$\mathfrak C_{\mathcal S}:= (\underset{ \mathcal{S}}{ \oplus}\mathfrak{C})\cap \mathfrak C$.

As in Lemma \ref{lem:union} we have
\begin{equation}\label{eq:global}
\mathfrak C= \{\bf1\} \cup \underset {\mathcal S\in Ind(\mathfrak C)_{(\geq 2)}}\bigcup  \mathfrak C_{\mathcal S}
\end{equation}
and
%\begin{equation}\label{eq:classeindecbis}\mathfrak C(\mathcal R)=\{\bf1\}\cup (\underset{\mathcal R'\in \mathfrak D_{(2)}\setminus \{\mathcal R \}} {\bigcup}\mathfrak C_{\mathcal R'})\cup \underset
%{\mathcal R\in \mathfrak D_{(\geq 3)}}\bigcup  \mathfrak C_{\mathcal R},
%\;~ \text{for all}\; \mathcal R\in\mathfrak D_{(2)}.\end{equation}
\begin{equation}\label{eq:classeindecbis}\mathfrak C(\mathcal S)=\mathfrak C \setminus \mathfrak C_{\mathcal S} \; \text{for every}\;  \mathcal S \in Ind(\mathfrak C)_{(2)}.\end{equation}

%Let $\{\mathcal S_1,\cdots,\mathcal S_p\}$  be an enumeration of $\mathfrak D_{(2)}\cap \mathfrak C$.

 Let $\mathcal H$ and  $\mathcal K$ be  the generating series of
$\mathfrak{C}$ and  $Ind(\mathfrak C)_{(\geq 3)}$ respectively.

Let  $\mathcal H_{\mathfrak C_{(2)}}$ and $\mathcal H_{\mathfrak C_{(\geq 3)}}$ be the generating series of $\mathfrak C_{(2)}:=\underset{\mathcal S\in Ind(\mathfrak C)_{(2)}} \bigcup {\mathfrak C_\mathcal S}$ and of $\mathfrak C_{(\geq 3)}:=\underset{\mathcal S\in Ind(\mathfrak C)_{(\geq 3)}} \bigcup {\mathfrak C_\mathcal S}$.

% Let  $p$ be the cardinality of $Ind(\mathfrak C)_{(2)}$.
%
%, resp. $\mathfrak {C}(\mathcal S_i)$, resp.
%, resp. $\mathcal H_i$, resp.
We have:
\begin{equation}\label{eq:global}
\mathcal  H_{\mathfrak C}= x + \mathcal H_{\mathfrak C_{(2)}}+\mathcal H_{\mathfrak C_{(\geq 3)}}\end{equation}
and

\begin{equation}\label{eq:classeindecter}\mathcal H_{\mathfrak C(\mathcal S)}=\mathcal H_{\mathfrak C} - \mathcal H_{\mathfrak C_{\mathcal S}} \; \text {for every}\; \mathcal S\in  \mathfrak D_{(2)}.\end{equation}

%\begin{equation}\label{eq:classeindecbis}\mathfrak C(\mathcal R)=\{\bf1\}\cup (\underset{\mathcal R'\in \mathfrak D_{(2)}\setminus \{\mathcal R \}} {\bigcup}\mathfrak C_{\mathcal R'})\cup \underset
%{\mathcal R\in \mathfrak D_{(\geq 3)}}\bigcup  \mathfrak C_{\mathcal R},
%\;~ \text{for all}\; \mathcal R\in\mathfrak D_{(2)}.\end{equation}

 We deduce that  $\mathcal  H_{\mathfrak C}$  and $\mathcal H_{\mathfrak C(\mathcal S)}$ are algebraic for every  $\mathcal S \in Ind(\mathfrak C)_{(2)}$, from the following claims that we will prove afterwards\begin{claim}\label{claim:claim1}The generating series  of $\mathcal H_{\mathfrak C_{(\geq 3)}}$ is  a polynomial in the generating series  $\mathcal H_{\mathfrak C}$  whose coefficients are  algebraic series.

\end{claim}
 \begin{claim}\label{claim:claim2}
%The generating series  of $\mathcal H_{\mathfrak C_{(2)}}$ is a quotient of two polynomials in
%the generating series  $\mathcal H_{\mathfrak C}$  whose coefficients are  algebraic series, the valuation of the denominator considered as a polynomial in $x$ being $0.$ \end{claim}
%
For each $\mathcal S \in Ind(\mathfrak C)_{(2)}$, the generating series $\mathcal H_{\mathfrak C(\mathcal S)}$ of ${\mathfrak C(\mathcal S)}$ is either a linear polynomial in
the generating series  $\mathcal H_{\mathfrak C}$  of the form \begin{equation} \label {eq:almostfinal2}
\mathcal H_{\mathfrak C(\mathcal S)}=\frac{(1-\alpha)\mathcal H_{\mathfrak C}-\delta}{1+\beta};
\end{equation} whose coefficients are  algebraic series or is a rational fraction of the form \begin{equation} \label {eq:almostfinal1}
\mathcal H_{\mathfrak C(\mathcal S)}=\frac{\mathcal H_{\mathfrak C}}{1+ \mathcal H_{\mathfrak C}}.
\end{equation}
\end{claim}

 Substituting in formula \ref{eq:global} the values of $\mathcal H_{\mathfrak C_{(\geq 3)}}$ and $\mathcal H_{\mathfrak C_{( 2)}}$ given by Claim \ref{claim:claim1} and Claim \ref{claim:claim2} we obtain a polynomial in
$\mathcal{H_{\mathfrak C}}$ whose coefficients are algebraic series. This polynomial is not identical to zero. Indeed, it is the sum of a polynomial $A=a_0+a_1\mathcal H_{\mathfrak C}+a_2\mathcal H_{\mathfrak C}^2$ and $B=b_0+b_1\mathcal H_{\mathfrak C}+\cdots+a_k\mathcal H_{\mathfrak C}^k$ whose coefficients are algebraic series (in fact, $B=\mathcal H_{\mathfrak C_{(\geq 3)}}(1+\mathcal H_{\mathfrak C})$).  The valuation of $A$ and $B$ as series in $x$ are distinct. Indeed, the valuation of $A$ is $1$ (notice that $a_0=x+\delta$ where $\delta$ is either zero or an algebraic series of valuation at least $2$).  Hence, if $B\neq 0$ (when $Ind(\mathfrak C)_{(\geq 3)}$ is non empty) its valuation is at least $3$. %$A+B$ is non zero, hence $\mathcal H_{\mathfrak C}$ is algebraic. If $B\neq 0$, the valuation of $B$, as a series in $x$ is $3$ at least $3$. The valuation of $A$, as polynomial in $x$ is at least $1$.
Since $A$ and $B$ don't have the same valuation, then $A+B$ is not identical to zero.   Being a solution of a non zero polynomial, $\mathcal H_{\mathfrak C}$ is algebraic. With this result and claim \ref{claim:claim2}, $\mathcal H_{\mathfrak C(\mathcal S)}$ is algebraic.  %each $\mathcal K_i=\frac{N_i}{D_i}$ with valuation of $D_i$ equal to zero, whereas the valuation of each of the
%other terms is non zero. Hence $\mathcal{H_{\mathfrak C}}$ is an algebraic series. It follows from Claim \ref{claim:claim2} that the $\mathcal K_i$ are also algebraic series.
With this, the proof of
Theorem \ref{theo:algebraic} is complete.
\bigskip

In order to prove our claims, we need the following lemmas (respectively Lemma 15 and Lemma 18 in \cite{A-A}).

\begin{lemma}\label{lem:intersection}
 Let $\mathcal S$ be an indecomposable ordered structure and $\mathfrak A:=(\mathfrak A_x)_{x\in \mathcal S}$, $\mathfrak B:=(\mathfrak B_x)_{x\in \mathcal S}$ be two sequences of subclasses of
ordered binary structures indexed by the elements of $\mathcal S$. If $\mathcal S$ has at least three elements then $$(\underset{x\in \mathcal S}{\oplus}{\mathfrak A_x})\cap
(\underset{x\in \mathcal S}{\oplus}{\mathfrak B_x})=\underset{x\in \mathcal S}{\oplus}({\mathfrak A_x\cap \mathfrak B_x}).$$
If $\mathcal S:=(\{0,1\},\leq,(\rho_i)_{i\in J})$ with $0<1$ then
$$(\mathfrak A_0(\mathcal S)\underset{\mathcal S}{\oplus}\mathfrak A_1)\cap (\mathfrak B_0(\mathcal S)\underset{\mathcal S}{\oplus}\mathfrak B_1)=(\mathfrak A_0(\mathcal S)\cap \mathfrak B_0(\mathcal S))
\underset{\mathcal S}{\oplus}(\mathfrak A_1\cap \mathfrak B_1).$$
\end{lemma}
\begin{proof}
The first equality follows from Proposition \ref{pro:sumlex} and the second one follows from Lemma \ref{lem:decomposition}.
\end{proof}

\bigskip

Let  $\mathfrak C$ be a  class of finite structures and $\overline {\mathcal B} := \mathcal B_1,\mathcal B_2,...,\mathcal B_l$ be a sequence of finite structures, we will set $\mathfrak C < \overline {\mathcal B}>:= \mathfrak C < \mathcal B_1,\mathcal B_2,...,\mathcal B_l>:=Forb(\{\mathcal B_1,\mathcal B_2,...,\mathcal B_l\})
\cap \mathfrak C.$

\noindent If $\mathfrak C$ is hereditary, a proper hereditary subclass $\mathfrak C'$ of
$\mathfrak C$ is \emph{strong} if every bound of $\mathfrak C'$ in $\mathfrak C$ is embeddable in some bound of $\mathfrak C.$ Note that the intersection of strong subclasses is strong.

\medskip
\noindent Let $\mathfrak A:=(\mathfrak A_x)_{x\in \mathcal S}$, where $\mathcal S$ is indecomposable with at least three elements. A \emph{decomposition of a binary structure} $\mathcal B$ \emph{over} $\mathfrak A$ is a map $h:\mathcal B\rightarrow\mathcal S$ such that
$\mathcal B=\underset{x\in \mathcal S \restriction_{range(h)}}{\oplus}{\mathcal B}
\restriction_{h^{-1}(x)}$ and $\mathcal B \restriction_{h^{-1}(x)} \in \mathfrak A_x$ for all $x\in range(h).$ Hence,  each $\mathcal B \restriction_{h^{-1}(x)}$ is an interval of $\mathcal B.$ Let $H_{\mathcal B}$ be the set of all such decompositions of $\mathcal B$.

\begin{lemma}\label{lem:egality}Let $\mathcal S$ be an indecomposable ordered structure, $\mathfrak A:=(\mathfrak A_x)_{x\in \mathcal S}$ be a sequences of subclasses of
ordered binary structures indexed by the elements of $\mathcal S$, $\overline {\mathcal B}:=\mathcal B_1,\mathcal B_2,...,\mathcal B_l$ be a sequence of finite structures, and $\mathfrak C:=(\underset{x\in \mathcal S}{\oplus}{\mathfrak A_x})<\overline {\mathcal B}>$. If $\mathcal S$ has at least three elements then $\mathfrak C$ is a union of sets of the form
$\underset{x\in \mathcal S}{\oplus}\mathfrak D_x$ where each $\mathfrak D_x$ is either $\mathfrak A_x<\overline{\mathcal B}>$ or one of its strong subclasses.
\end{lemma}

\begin{proof}
We  prove the result for $l=1$ and we set $\mathcal B:= \mathcal B_1$. For that we prove that:
\begin{equation} \label{eq:eqstrong}
\mathfrak C=\underset{h\in H_{\mathcal B}}{\bigcap}\;\underset{x\in range(h)}{\bigcup}\underset{y\in \mathcal S}{\oplus}{\mathfrak A^{(x)}_y}
\end{equation}
 where $\mathfrak A^{(x)}_x:=\mathfrak A_x<\mathcal B\restriction_{h^{-1}(x)}>$ and $\mathfrak A^{(x)}_y:=\mathfrak A_y$ for $y\neq x$.

Let's call by $(1)$ (respectively by $(2)$) the left-hand side (respectively the right-hand side) of Equation \ref{eq:eqstrong}. Inclusion $(1)\subseteq (2)$ holds without any assumption. Indeed,
let $\mathcal T$ in $(1)$. We prove that $\mathcal T$ is in $(2)$. If $h$ is a decomposition of $\mathcal B$, we want to find $x\in range(h)$ such that $\mathcal T\in \underset{y\in \mathcal S}{\oplus}
\mathfrak A^{(x)}_y.$
Since $\mathcal T$ is in $(1)$, it has a
decomposition over $\mathcal S.$  Let $h\in H_{\mathcal B}$, since $\mathcal B \nleq \mathcal S$, there exist $x\in \mathcal S\restriction_{range(h)}$ such that $\mathcal B\restriction_{h^{-1}(x)}\nleq
\mathcal T,$ hence
$\mathcal T\in \underset{y\in \mathcal S}{\oplus}\mathfrak A^{(x)}_y.$

\noindent Inclusion $(2)\subseteq (1)$ holds under the assumption that a structure in $(2)$ has a unique decomposition over $\mathcal S$ and that it is ordered, (what means that $\mathcal S$ is rigid, that is $\mathcal S$
has no automorphism distinct from the identity).
Let $\mathcal T$ in $(2)$, then for every $h\in H_{\mathcal B}$ there exist $x_h\in range (h)$ such that $\mathcal T\in \underset{y\in \mathcal S}{\oplus}{\mathfrak A^{(x_h)}_y}$, thus, $\mathcal T\in \underset{h\in H_{\mathcal B}}{\bigcap}\underset
{y\in \mathcal S}{\oplus}{\mathfrak A^{(x_h)}_y}.$
We have $\mathcal T\in \underset{y\in \mathcal S}{\oplus}{\mathfrak A_y}$ because, $\mathfrak A^{(x_h)}_y\subseteq \mathfrak A_y$ for every $h.$ Hence, $\mathcal T=\underset{y\in \mathcal S}{\oplus}\mathcal T_y.$
 We claim that $\mathcal B\nleq \mathcal T$. Suppose $\mathcal B\leq \mathcal T$
let $f$ be an embedding of $\mathcal B$ into $\mathcal T$ and  $h:=pof$, where $p$ is the projection map from $\mathcal T$ into $\mathcal S,$ we must have $\mathcal B\restriction_{h^{-1}(x)}\leq \mathcal T_x$ for $x\in rang(h)$ which is a contradiction with the fact that
$\mathcal T\in \underset{h\in H}{\bigcap}\underset{y\in \mathcal S}{\oplus}{\mathfrak A^{(x_h)}_y}.$

\noindent Using distributivity of intersection over union, we may write $(2)$ as a union of terms, each of which is an intersection of terms like $\underset{y\in \mathcal S}{\oplus}\mathfrak A_y<\mathcal B_x>$, where $\mathcal B_x$ is an interval of $\mathcal B$ such that, there exist a decomposition $h$ of $\mathcal B$ and $\mathcal B_x=\mathcal B \restriction_{h^{-1}(x)}.$
These intersections, by lemma \ref{lem:intersection} and the fact that among all decompositions of $\mathcal B$ are all ones which send $\mathcal B$ into a single element $x$ of $\mathcal S$ have the form
$\underset{x\in \mathcal S}{\oplus}{\mathfrak D_x}$ where each $\mathfrak D_x$ is of the form
$\mathfrak A_x<\mathcal B,\cdots>$ where the structures occurring after $\mathcal B$ (if any) are intervals of $\mathcal B.$ Hence, $\mathfrak D_x$ is either $\mathfrak A_x<\mathcal B>$ or one of its strong subclasses. The case $l>1$ follows by induction.

%The proofs of these lemmas rely on two facts: a) if $\mathcal T$ is a sum over $\mathcal R$, then the decomposition of $\mathcal T$ into intervals is unique. b) $\mathcal R$ is rigid, that is $\mathcal R$
%has no automorphism distinct from the identity.
\end{proof}

\noindent{\bf Proof of Claim \ref{claim:claim1}.}
Since $\mathfrak  A$ is wqo and $\mathfrak C$ is a proper hereditary subclass, we have $\mathfrak C=\mathfrak A <\overline {\mathcal B} >$ for some finite family
$\overline {\mathcal B}:= \mathcal B_1,\mathcal B_2,...,\mathcal B_l$ of elements of $\mathfrak A$. Let $\mathcal S \in Ind(\mathfrak C)_{(\geq 3)}$,   Lemma \ref{lem:egality} asserts that  $\mathfrak C_{\mathcal S}$ is an union
of classes, not necessarily disjoint,    of the form  $\underset{x\in \mathcal S}{\oplus}\mathfrak C_x$ where each $\mathfrak C_x$ is either $\mathfrak C$ or one of its strong subclasses.
The generating series of$\underset{x\in \mathcal S}{\oplus}\mathfrak C_x$ is a monomial in the generating series $\mathcal H_{\mathfrak C}$ of  $\mathfrak C$ whose coefficient  is a product of generating
series of proper strong
subclasses of $\mathfrak C$. From the induction hypothesis, the generating series of  of these strong subclasses are algebraic series, hence this coefficient is an algebraic series.  Using the principle of inclusion-exclusion, we get that the generating series
$\mathcal H_{\mathfrak C_{\mathcal S}}$ of $\mathfrak C_{\mathcal S}$ is a polynomial in the generating series  $\mathcal H_{\mathfrak C}$  whose coefficients  are algebraic series. Since the
  $\mathfrak C_{\mathcal S}$ 's are pairwise disjoint, the generating series $\mathcal H_{\mathfrak C_{(\geq 3)}}$
is also a polynomial in the generating series  $\mathcal H_{\mathfrak C}$  whose coefficients are  algebraic series.

 \hfill $\Box$

\begin{lemma}\label{lem:twoelements}
If $\mathcal S$ has two elements $0$ and  $1$, $\mathcal S:=(\{0,1\},\leq,(\rho_i)_{i\in J})$ with $0<1$,   then $(\mathfrak A(\mathcal S)\underset{\mathcal S}{\oplus}\mathfrak A)<\overline {\mathcal B}>$ is an union of classes of the form
$({\mathfrak A'(\mathcal S)<\overline{\mathcal B}>})\underset{\mathcal S}{\oplus}(\mathfrak A''<\overline {\mathcal B}>)$, where $\mathfrak A'<\overline{\mathcal B}>$ and $\mathfrak A''<\overline{\mathcal B}>$ are  either  equal to
$\mathfrak A<\overline{\mathcal B}>$ or to some strong subclasses of $\mathfrak A<\overline{\mathcal B}>$.
\end{lemma}

\begin{proof}
As above we suppose first $l=1$. Equation \ref{eq:eqstrong} yields
\begin{equation} \label{eq:eqbinary}
\big(\mathfrak A(\mathcal S)\underset{\mathcal S}{\oplus}\mathfrak A\big)< {\mathcal B}>=\underset{h\in H_{\mathcal B}}{\bigcap} \left[\big(\mathfrak A(\mathcal S)<\mathcal B\restriction_{h^{-1}(0)}>
\underset{\mathcal S}{\oplus}\mathfrak A) \bigcup  (\mathfrak A(\mathcal S)\underset{\mathcal S} {\oplus}(\mathfrak A<\mathcal B\restriction_{h^{-1}(1)}>))\right]
\end{equation}
An induction take care of the case $l>1$.
\end{proof}

 \noindent{\bf Proof of Claim \ref{claim:claim2}.}
 Let $\mathcal S \in Ind(\mathfrak C)_{(\geq 3)}$.   Lemma \ref{lem:twoelements} asserts that  $\mathfrak C_{\mathcal S}$ is an union of classes, not necessarily disjoint,
of the form ${\mathfrak C'(\mathcal S)}\underset{\mathcal S}{\oplus}\mathfrak C''$, where $\mathfrak C'$ and $\mathfrak C''$ are either  equal to $\mathfrak C$ or to some strong subclasses of $\mathfrak C$.
The generating series of these classes  are of the form $\mathcal H_{\mathfrak C(\mathcal S)}\mathcal H_{\mathfrak C}$ or $\alpha \mathcal H_{\mathfrak C}$ or $\beta \mathcal H_{\mathfrak C(\mathcal S)}$,
where $\alpha$ and $\beta$ are algebraic series. Using the principle of inclusion-exclusion, we get that the generating series $\mathcal H_{\mathfrak C_{\mathcal S}}$ is either of the form
$\mathcal H_{\mathfrak C(\mathcal S)}\mathcal H_{\mathfrak C}$ or of the form $\alpha \mathcal H_{\mathfrak C}+ \beta \mathcal H_{\mathfrak C(\mathcal S)}+\delta$,
where $\alpha, \beta$ and $\delta$ are algebraic series. In particular   $\mathcal H_{\mathfrak C_{\mathcal S}}$ is of the form $\alpha_{\mathcal S} \mathcal H_{\mathfrak C(\mathcal S)}+ \beta_{\mathcal S}$
 where $\alpha_{\mathcal S}$ and $\beta_{\mathcal S}$ are polynomials in $\mathcal H_{\mathfrak C}$ of degree at most $1$ with algebraic series as coefficients. Using  Equation \ref{eq:classeindecbis} we obtain
\begin{equation} \label {eq:almostfinal1}
\mathcal H_{\mathfrak C(\mathcal S)}=\frac{\mathcal H_{\mathfrak C}}{1+ \mathcal H_{\mathfrak C}};
\end{equation}
when all bounds  $\mathcal {B}_i$ of $\mathfrak C$ in $\mathfrak A$ are  $\mathcal S$-indecomposable or

\begin{equation} \label {eq:almostfinal2}
\mathcal H_{\mathfrak C(\mathcal S)}=\frac{(1-\alpha)\mathcal H_{\mathfrak C}-\delta}{1+\beta};
\end{equation}
if at least one bound  $\mathcal B_i$ is not $\mathcal S$-indecomposable.

\hfill $\Box$

%According to Proposition \ref{pro:sumlex} and Lemma \ref{lem:union}, the $\mathfrak C_{\mathcal R}$ 's form a partition of $\mathfrak C_({\geq 2})$. Since there are only finitely many $\mathcal R$ the
%series $\mathcal H_{\mathfrak C'}-x$ is a polynomial  in  $\mathcal H_{\mathfrak C'}$ with some algebraic series as coefficients. It follows that  $\mathcal H_{\mathfrak C'}$ is algebraic.
%This provides a contradiction and complete the proof of Theorem \ref{theo:algebraic}.
%
%. This set decomposes into a disjoint union of sets, some are of the form $\underset{x\in \mathcal R}{\oplus}\mathfrak A_x<\mathcal S_1,...,\mathcal S_k>$
%where $\mathfrak A_x=\mathfrak A$ for all $x\in \mathcal R.$ As in \cite{A-A}, using lemma \ref{lem:egality}, remark \ref{remark} and distributivity of intersection over union in the expression, we obtain a union, not necessarily disjoint, of terms each of which is  like $\underset{x\in \mathcal R}{\oplus}\mathfrak D_x$ where each $\mathfrak D_x$ is either $\mathfrak C'$ or one of its strong subclasses.
%
%

%\begin{remark}\label{rem:cograph}
The conclusion of Theorem \ref{theo:algebraic} above does not hold with structures which are not necessarily ordered.

\begin{example}

Let  $K_{\infty,\infty}$ be the direct sum of  infinitely many copies of the complete graph on an infinite set.  As it is easy to see the generating function of $Age(K_{\infty,\infty})$ is the integer partition function. This  generating series is not algebraic. However, $Age(K_{\infty,\infty})$ contains no indecomposable member with more than two elements. More generally, note that the class $Forb(P_4)$ of finite cographs contains no indecomposable cograph with more than two vertices and that this class is  not hereditary algebraic. Finite cographs are comparability graphs of serie-parallel posets which in turn are intersection orders of separable bichains.  By Albert-Atkinson's theorem, the class of these bichains  is hereditary algebraic. This tells us that algebraicity is not necessarily preserved by the transformation of a class into an other via a process as above (
 processes of this type are the  free-operators of Fra\"{\i}ss\'e \cite{fraisse}).
\end{example}

\section{A conjecture and some questions}\label{sect:conjecture}
In their paper \cite{A-A}, Albert and Atkinson indicate  that there are  infinite sets of simple  permutations whose sum closure is algebraic  but,  as it turns out,  some hereditary  subclasses are not necessarily algebraic. An example is the collection of decreasing oscillations (see   the end of the section). In order to extend their proof to some other classes, they ask  whether there exists \emph{an infinite set of simple permutations whose sum-closure is well quasi
ordered}. As we indicate in Proposition \ref{criticalwqo} below, the set of exceptional permutations has this property. In fact,  it is hereditary wqo.
We guess that this notion of hereditary wqo is the right concept for extending Albert-Atkinson theorem.

\noindent Exceptional permutations correspond to bichains which are \emph{critical} in the sense of Schmerl and Trotter. Let us recall that a binary structure $\mathcal R$
with domain $E$ is \emph{critical} if $\mathcal R$ is
indecomposable but $\mathcal R\restriction_{E\setminus  \{x\}}$ is not indecomposable for every $x\in E$. Schmerl and Trotter  \cite{S-T} gave a description of
critical posets. They fall into two infinite classes:
$\mathfrak{P}:=\{\mathcal P_n: n\in \mathbb N\}$ and $\mathfrak{P'}:=\{\mathcal P'_n: n\in \mathbb N\}$ where
$\mathcal P_n:=(V_n,\leq_n)$, $V_n:=\{0, \dots, n-1\}\times \{0,1\}$,  $(x,i)<_n(y,j)$ if  $i<j \; \text{and} \; x\leq y$;
$\mathcal P'_n:=(V_n,\leq'_n)$   and $(x,i)<'_n (y,j)$ if $ j\leq i \; \text{and} \; x< y$.

\noindent These posets are two-dimensional. That is, they  are intersection of two linear orders which are respectively $L_{n,1}:=(0,0)<(0,1)<\cdots<(i,0)<(i,1)\cdots <(n-1,0)<(n-1,1)
$ and $L_{n,2}:=(n-1,0) <\cdots<(n-i,0) <\cdots <(0, 0)< (n-1,1)<\cdots  <(n-i,1)\cdots <(0, 1)$ for $\mathcal P_n$ and $L'_{n,1}:=L_{n,1}$
and $L'_{n,2}:=(L_{n,2})^*$  for $\mathcal P'_n.$\medskip

\noindent As it is well known, an indecomposable two-dimensional poset $\mathcal P:=(V, L)$ has a unique realizer (that is there is a  unique pair  $\{L_1,L_2\}$ of linear
orders whose intersection is the order $L$ of $\mathcal P$). Hence,
there are at most two bichains, namely  $(V, L_1,L_2)$ and $(V, L_2,L_1)$ such that $L_1\cap L_2=L$. The  critical posets described above yield four kind
of bichains, namely
$(V_n,L_{n,1},L_{n,2})$, $(V_n,L_{n,2},L_{n,1})$, $(V_n,L_{n,1},(L_{n,2})^*)$ and $(V_n,(L_{n,2})^*,L_{n,1})$. These bichains are critical. Indeed, a bichain is indecomposable if and only if the intersection order
is indecomposable  (\cite{R-Z} for finite bichains and \cite {zaguia} for  infinite bichains).  The isomorphic types of these bichains are described in Albert and Atkinson's paper in terms of permutations of
$1, \dots, 2m$ for $m\geq 2$:\\
$(i)$~~$2.4.6....2m.1.3.5....2m-1.$\\
$(ii)$~~$2m-1.2m-3....1.2m.2m-2....2.$\\
$(iii)$~~$m+1.1.m+2.2....2m.m.$\\
$(iv)$~~$m.2m.m-1.2m-1....1.m+1.$\\

For example, the type of the bichain  $(V_m, L_{m,1}, L_{m,2})$ is  the permutation given in $(iv)$, whereas  the type of $(V_m, L_{m,2}, L_{m,1})$ is its inverse, given in $(ii)$
(enumerate the elements of $V_m$ into the sequence $1, \dots, 2m$, this according to the order $L_{m,1}$, then reorder this sequence according to the order $L_{m,2}$; this  yields the sequence
$\sigma^{-1}:= \sigma^{-1}(1),\dots, \sigma^{-1}(2m)$; according to our definition the type of $(V_m, L_{m,1}, L_{m,2})$ is the permutation $\sigma$, this is the one given in $(iv)$).
For $m=2$, the permutations given in $(i)$ and $(iv)$ coincide with $2413$ whereas those  given in $(ii)$ and $(iii)$ coincide with $3142$; for larger values of $m$, they are all different.

The four classes of indecomposable bichains are obtained from $\mathfrak B:=\{(V_n, L_{n,1}, L_{n,2}): n\in \NN\}$  by exchanging  the two orders in each bichain or by reversing the order of the first one, or by reversing the second one.   Hence the order structure w.r.t. embedabbility of these classes is the same, and it remains the same if we label the elements of these bichains.

\begin{proposition}\label{criticalwqo}
 The class of critical bichains is hereditary wqo.
\end{proposition}

\begin{proof}
This class is the union of four classes  hence, in order to prove that it is hereditary wqo, it suffices to prove that each one of these classes is hereditary wqo. According to the observation above, it suffices to prove that one, for example $\mathfrak B$, is hereditary wqo. Let $\mathcal A$ be a wqo poset. We have to prove that $\mathfrak B. \mathcal A$ is wqo. For that, set $\mathcal B:=\mathcal A^2$, where $\mathcal A^2:=\{e:\{0,1\}\rightarrow \mathcal A\}$, and order $\mathcal B$ componentwise. Let $\mathcal B^*$ be the set of all words over the ordered alphabet $\mathcal B$. We define an order preserving map $F$ from $\mathcal B^*$ onto $\mathfrak B. \mathcal A$. This will suffice. Indeed, $\mathcal B$ is  wqo as a product of two wqo sets;  hence,  according to Higman theorem on words over ordered alphabets \cite{higman}, $\mathcal B^*$ is wqo. Since $\mathfrak B. \mathcal A$ is the image of a wqo by an order preserving map,  it is wqo. We define the map  $F$ as follows.  Let $w:=w(0)w(1)\cdots w(n-1)\in \mathcal B^*$. Set $F(w):=(\mathcal R, f_w)\in \mathfrak B. \mathcal A$ where $\mathcal R:= (V_n,  L_{n,1}, L_{n,2})$ and $f_w(i,j):=w(i)(j)$ for $j\in \{0,1\}$.  We observe first that $w\leq w'$ in $\mathcal B^*$ implies $F(w)\leq F(w')$ in $\mathfrak B. \mathcal A$. Indeed, if $w\leq w'$ there is an embedding $h$ of the chain $0<\cdots<n-1$ into the chain $0<\cdots<n'-1$ such that $w(i)\leq w'(h(i))$  for all $i<n$. Let $\overline h:\{0,\dots, n-1\}\times \{0,1\}\rightarrow \{0,\dots, n'-1\}\times \{0,1\}$ defined by setting $\overline{h}(i,j):=(h(i), j)$. As it is easy to check,  $\overline h$ is an embedding of $F(w)$ into  $F(w')$. Next, we note that $F$ is surjective. Indeed, if $(\mathcal R, f)\in \mathfrak B. \mathcal A$ with  $\mathcal R:= (V_n,  L_{n,1}, L_{n,2})$, then the word $w:=w(0)w(1)\cdots w(n-1)$ with $w(i)(j):=f(i,j)$ yields $F(w)=(\mathcal R, f)$. \end{proof}

\noindent With Proposition \ref{prop:wqo}, we have:

\begin{corollary}
 The sum-closure of the class of critical bichains is wqo.
\end{corollary}

\noindent In \cite{A-A} it is mentioned that this class has finitely many bounds.
The generating series of the class of critical bichains is rational (the class is covered by four chains). According to Corollary 13 of \cite {A-A} their sum-closure is algebraic.

\begin{question}
%\begin{enumerate}
%\item
Is the sum-closure of the class of critical bichains  hereditary algebraic?
%\item More generally, is a hereditary class of bichains algebraic provided that  it is wqo?
%\end{enumerate}
\end{question}

We conjecture that the answer is positive. This will be a consequence of a conjecture for hereditary classes of ordered binary structures that we formulate below.

%  Indeed, the class
%of critical bichains satisfies the hypotheses of this conjecture.

\begin{conjecture}\label{conjec}
If $\mathfrak D$ is a hereditary class of indecomposable ordered binary structures which is hereditary wqo and hereditary algebraic, then its sum-closure is hereditary algebraic.
\end{conjecture}

\noindent The requirement  that $\mathfrak D$ is wqo will not suffice in Conjecture \ref{conjec}.

 Indeed, let  $\mathcal P_{\ZZ}$ be the doubly infinite path whose vertex set is $\ZZ$ and edge set  $E:=\{(n,m)\in \ZZ\times \ZZ: \vert n-m\vert=1\}$. The edge set $E$ has two transitive orientations,
e.g. $P:=\{(n,m)\in \ZZ\times \ZZ: \vert n-m\vert=1 \; \text{and}\; n\;  \text {is  even}\}$ and its dual $P^*$. As an order, $P$ is the intersection of the linear orders $L_1:=\cdots<2n<2n-1<2(n+1)<2n+1<\cdots$
and $L_2:=\cdots<2(n+1)<2n+3<2n<2n+1<\cdots.$ Let $\mathcal C:= (\ZZ, L_1,L_2)$ and $\mathfrak D:=Ind(\mathcal C)$.

 \begin{lemma}  $\mathfrak D$ is wqo but not hereditary wqo.
 \end{lemma}

 \begin{proof}Members of $\mathfrak D$ of size $n$ are obtained by restricting $\mathcal C$ to intervals of size $n$, $n\not =3$,  of the chain $(\ZZ, \leq)$ (observe first that the graph
$P_\ZZ$ is indecomposable as  all its restrictions to intervals of size different from $3$ of the chain $(\ZZ, \leq)$ and furthermore there are no others indecomposable restrictions; next,
use the fact that the indecomposability of a comparability graph amounts to the indecomposability of its orientations \cite{Ke}, and that the indecomposability of a two-dimensional poset amounts
to the indecomposability of the bichains associated with the order \cite {zaguia}).
Up to isomorphy, there are two indecomposable bichains of size $n$, $n\not=3$, namely
$\mathcal C_n:=\mathcal C_{\restriction \{0, \dots, n-1\}}$ and $\mathcal {C}_n^*:=\mathcal C^*_{\restriction \{0, \dots, n-1\}}$ where $\mathcal C^*:= (\ZZ, {L^*}_1,{L^*}_2)$.
These two bichains  embed all members of $\mathfrak D$ having size less than $n$. Being covered by two chains, $\mathfrak D$ is wqo. To see that $\mathfrak D$ is not hereditary wqo, we may
associate  to each  indecomposable member of $\mathfrak D$  the comparability graph of the intersection of the two orders and observe that this association  preserves the embeddability relation,
even  though label are added.  The  class of graphs obtained from this association consists of  paths of size distinct from $3$. It is not hereditary wqo. In fact, as it is immediate to see,
if a class  $\mathfrak G$ of graphs  contains infinitely many paths of distinct sizes, then $\mathfrak G. 2$  is  not wqo. Indeed, if we label  the end vertices of each path by $1$ and label the
other vertices  by $0$,  we obtain an infinite antichain. Thus $\mathfrak D.2$ is not wqo. \end{proof}

%To see that, it suffices to label members of $\mathfrak D$ by $\underline{2}$, the poset made of $\{0,1\}$ ordered so that $0<1$, in the following manner:
%\begin{itemize}
%\item if the domain of the structure is of form $V_i$, then associate the map $h_1:V_i\rightarrow \{0,1\}$ such that $h_1((i+1,1))=h_1((i+\dfrac{n}{2},0))=1$ and $h_1(x)=0$ for all
%$x\in V_i\setminus \{(i+\dfrac{n}{2},0),(i+1,1)\}.$
%\item if the domain of the structure is of form $V'_i$, then associate the map $h_2:V'_i\rightarrow \{0,1\}$ such that $h_2((i+1,0))=h_2((i+1+\dfrac{n}{2},1))=1$ and $h_2(x)=0$ otherwise.
%\item if the domain of the structure is of form $W_i$, then associate the map $h_3:W_i\rightarrow \{0,1\}$ such that $h_3((i+1,1))=h_3((i+\dfrac{n+1}{2},1))=1$ and $h_3(x)=0$ otherwise.
%\item if the domain of the structure is of form $W'_i$, then associate the map $h_4:W'_i\rightarrow \{0,1\}$ such that $h_4((i+1,0))=h_4((i+\dfrac{n+1}{2},0))=1$ and $h_4(x)=0$ otherwise.
%\end{itemize}
%

The generating series of $\mathfrak D$ is rational (its generating function is $\dfrac{x+x^2}{1-x}$).  In fact, $\mathfrak D$ is hereditary algebraic (every hereditary subclass of $\mathfrak D$ is finite).
By Corollary 13 of \cite {A-A},  the sum-closure $\sum \mathfrak D$ is algebraic. (in fact, if  $D$ is the generating function of $\sum \mathfrak D,$ then $2D^5+2D^4-D^3+(2-x)D^2-D+x=0.$).
But $\sum \mathfrak D$ is not hereditary algebraic. For that,  it suffices to observe that it is not  wqo  and to  apply  Lemma \ref{lem:wqo}. The fact that $\sum \mathfrak D$ is not wqo is because
we may embed the poset $\mathfrak D.\underline{2}$ into $\sum \mathfrak D$ via an order preserving map. A simpler argument consist to observe first that the family $(G_n)_{n\in \NN}$, where $G_n$
is the graph obtained from the $n$-vertex path $P_n$ by replacing its end-vertices by a two-vertex independent set, is an antichain, next that these graphs are comparability graphs associated to members
of $\mathfrak D$.

The permutations corresponding to the members of $\mathfrak D$ are called \emph{decreasing oscillations}. They have been the object of several studies:

The downward closure
$\downarrow\mathfrak D$ is $Age (\mathcal C)$,  the age of $\mathcal C$; this age has four obstructions,  it is  rational: the generating series is $\dfrac{1-x}{1-2x-x^{3}}$,
the generating function being  the sequence A05298 of \cite {Sloane}, starting by 1, 1, 2, 5, 11, 24. For all of this  see \cite{brignall-al}.

\subsection{Questions.}
Is it true that:
\begin{enumerate}

%\item  the number of hereditary classes of bichains which are wqo  is countable?
%
%\item a hereditary class of bichain is  hereditary wqo if and only if it does not contain $Age (\mathcal C)$, where $\mathcal C:= (\ZZ, {L}_1,{L}_2)$?

\item a hereditary class of indecomposable ordered binary structures $\mathfrak D$ is hereditary wqo whenever its sum closure   is hereditary algebraic?
\item  the generating series of a hereditary class of relational structures is rational whenever the profile of this class is bounded by a polynomial? This is true for graphs \cite{B-B} and tournaments \cite{B-P}.

\item the profile of a wqo hereditary class of relational structures is bounded above by some exponential?

\end{enumerate}


\begin{thebibliography}{9}

\bibitem{A-A}
M.H.~ Albert and M.D. Atkinson,
\newblock {\em Simple permutations and pattern restricted permutations.}
\newblock {\em Discrete Mathematics}, {\bf 300} (2005) 1--15.

\bibitem{A-A-K}
M.H.~ Albert, M.D. Atkinson and M. Klazar,
\newblock {\em The enumeration of simple permutations.}
\newblock  Journal of integer sequences, {\bf 6} (2003), Article 03.4.4.

\bibitem{A-A-V}
M.H.~ Albert, M.D. Atkinson and V. Vatter,
\newblock {\em Subclasses of the separable permutations.}
\newblock {\em Bull. Lond. Math. Soc.} {\bf 43} (2011), no. 5, 859--870.
\bibitem{B-B}
J.~ Balogh, B. Bollob\'{a}s, M. Saks and V. T. S\'{o}s,
\newblock{\em The unlabelled speed of a hereditary graph property.}
\newblock{\em Journal of combinatorial theory,} {ser B} {\bf 99} (2009) 9--19.

\bibitem{B-B-2}
J.~ Balogh, B. Bollob\'{a}s and R. Morris,
\newblock{\em Hereditary properties of combinatorial structures: posets and oriented graphs.}
\newblock{\em Journal of Graph Theory,} {\bf 56} (2007) 311--332.

\bibitem{B-B-3}
J.~ Balogh, B. Bollob\'{a}s and R. Morris,
\newblock{\em Hereditary properties of partitions, ordered graphs and ordered hypergraphs.}
\newblock{\em European Journal of combinatorics,} {\bf 8} (2006) 1263--1281.

\bibitem{B-P}
Y.~Boudabous and M. Pouzet,
\newblock {\em The morphology of infinite
tournaments; application to the growth of their profile.} \newblock{\em  European Journal
of Combinatorics.} {\bf 31} (2010) 461-481.

\bibitem{Br}
R.~ Brignall,
\newblock {A survey of simple permutations. Permutation patterns.}
\newblock 41-65, London Math. Soc. Lecture Note Ser., 376, Cambridge Univ. Press, Cambridge, 2010.

\bibitem{brignall-al} R.~Brignall, N.~ Ruskuc, V.~ Vatter, Vincent Simple permutations: decidability and unavoidable substructures. Theoret. Comput. Sci. {\bf 391} (2008), no. 1-2, 150163.

\bibitem{cameron} P. J~Cameron, \newblock{Homogeneous permutations. Permutation patterns} (Otago, 2003). Electron. J. Combin. {\bf 9} (2002/03), no. 2, Research paper 2, 9 pp.
\bibitem{c-d}
B.~Courcelle, C. Delhomm\'e,
\newblock {The modular decomposition of countable graphs. Definition and
construction in monadic second-order logic}
\newblock Theoretical Computer Science {\bf 394} (2008) 1-38.

\bibitem{ehren}
A.~ Ehrenfeucht, T. Harju, G. Rozenberg,
 \newblock{\em The theory of 2-structures. A framework for decomposition and transformation of graphs.}
\newblock{\em World Scientific Publishing Co., Inc., River Edge, NJ, 1999.}

\bibitem{fraisse}
M.~R. Fra\"{\i}ss\'{e},
\newblock {\em Theory of relations}.
\newblock Second edition, North-Holland Publishing Co., Amsterdam, 2000.

\bibitem{fraisse2}
M.~R. Fra\"{\i}ss\'{e},
\newblock {\em On a decomposition of relations which generalizes the sum of ordering relations.}
\newblock Bull. Amer. Math. Soc., {\bf 59}  (1953) 389.

\bibitem{fraisse3}
M.~R. Fra\"{\i}ss\'{e},
\newblock {\em L'intervalle en th\'eorie des relations; ses g\'en\'eralisations; filtre intervallaire et cl\^oture d'une relation. (French)
 [The interval in relation theory; its generalizations; interval filter and closure of a relation].}
\newblock Orders: description and roles. (L'Arbresle, 1982), 313--341, North-Holland Math. Stud., 99, North-Holland, Amsterdam, 1984.

\bibitem{gallai}
T.~ Gallai,
\newblock {\em Transitiv orientbare graphen}.
\newblock Acta Math. Acad. Sci. Hungar. {\bf 18} (1967) 25-66 (English translation by F. Maffray and M. Preissmann in J.J. Ramirez-Alfonsin and B. Reed (Eds), Perfect graphs, Wiley 2001, pp.25-66.
\bibitem{higman} G.~Higman.
\newblock Ordering by divisibility in abstract algebras.
\newblock {\em Proc. London Math. Soc.} {\bf 3} (1952) 326--336, .

\bibitem{Ke}  D.~Kelly. \emph{Comparability graphs.} Graphs and Orders, I.
Rival (ed), NATO ASI Series, Vol.147, D. Reidel, Dordrecht, 1985,
pp. 3--40.
\bibitem{K-K}
T.~ Kaizer, M.~Klazar \newblock {\em On growth rates of closed permutation classes. Permutation patterns.}
\newblock {\em (Otago, 2003) electr. J. Combin.}, {\bf 9} (2002/2003), no. 2, Research Paper 10, 20 pp.

\bibitem{klazar}
M.~Klazar, \newblock Overview of general results in combinatorial enumeration, in {\em Permutation patterns},
London Math. Soc. Lecture Note Ser.,
Vol. 376, (2010), 3--40, Cambridge Univ. Press, Cambridge.
\bibitem{laflamme}
C.~Laflamme, M.~Pouzet, N.~Sauer, I.~Zaguia, \newblock Pairs of orthogonal countable ordinals,  15pp. 2014, to appear in J. Discrete Math.

\bibitem{Lo}  M. Lothaire,  \emph{Finite and Infinite Words}. Algebraic Combinatorics on Words. Cambridge University Press. 2002.
\bibitem{Marcus}
A.~ Marcus, G.~Tard\"os,   \newblock {\em Excluded permutation matrices and the Stanley-Wilf conjecture}, {\em J. Combin. Theory}, {Ser. A} {\bf 107} (2004), 153--160.

\bibitem{nosaki}
 A.~Nozaki, M.~Miyakawa, G.~Pogosyan, I.G.~Rosenberg, \newblock {\em The number of orthogonal permutations}, {\em Europ. J. Combinatorics}, {\bf 16} (1995) 71--85.

\bibitem{oudrar-pouzet}D.~Oudrar, M.~Pouzet, \newblock {\em Profile and hereditary classes of relational structures}, ISOR'11, International Symposium on Operational Research, Algiers ,
Algeria , May 30-June 2, 2011, H.Ait Haddadene,
I.Bouchemakh, M.Boudhar, S.Bouroubi (Eds)LAID3.

\bibitem{pouzet 72}
M.~Pouzet,
 Un belordre d'abritement et ses rapports avec les bornes d'une multirelation.
{\em Comptes rendus Acad. Sci. Paris,} S\'er A {\bf 274} (1972), pp.~1677--1680.


\bibitem{pouzet.tr.1978}
M.~Pouzet,
\newblock {\em Sur la th\'eorie des relations,}
\newblock  Th\`ese d'\'Etat, Universit\'e Claude-Bernard, Lyon 1,
  1978.

\bibitem{pouzet}
M.~ Pouzet,
\newblock {\em The profile of relations.}
\newblock {\em Glob. J.Pure Applied Math.} (Proceedings of the 14$
^{th}$ symposium of the Tunisian Mathematical Society, held in Hammamet, March 20-23, 2006), {\bf 2} (2006) 237--272.

\bibitem{C-R} C.~Retenauer,
Personnal communication, Nov.2011.
\bibitem{R-Z} I.~Rival, N.~Zaguia. \emph{ Perpendicular orders}.  Discrete Math.  {\bf 137}  (1995),  no. 1-3, 303--313.
\bibitem{S-T}
 J.H.~Schmerl, W.T.~Trotter,  \newblock {\em Critically indecomposable partially ordered sets, graphs, tournaments and other binary relational structures}, {\em Discrete Math.}, {\bf 113} (1-3) (1993) 191--205.

\bibitem{Sloane} N. J. A. Sloane, The On-Line Encyclopedia of Integer Sequences, sequence A111111.
\bibitem{zaguia}  I.~Zaguia. \emph{Prime two-dimensional orders and
perpendicular total orders.} Europ. J. of Combinatorics \textbf{19}
(1998), 639--649.
\end{thebibliography}
\end{document}